\documentclass[reqno,oneside,12pt]{amsart}

\usepackage[english]{babel}
\usepackage[utf8]{inputenc}
\usepackage{amsfonts,amsmath,amsthm}
\usepackage{stmaryrd}
\usepackage{thmtools,thm-restate}
\usepackage[T1]{fontenc}
\usepackage{enumitem}
\usepackage{varioref}
\usepackage[all]{xy}
\usepackage{units}
\usepackage{hyperref}
\usepackage{graphicx}
\usepackage{amssymb}
\usepackage{mathabx}
\usepackage{times,mathptm, mathtools}
\usepackage{etoolbox}
\usepackage{tikz-cd}
\usepackage{mathrsfs}
\usepackage{mathdots}
\usepackage[left=3cm, right=3cm]{geometry}

\docsvlist{A,B,C,D,E,F,G,H,I,J,K,L,M,N,O,P,Q,R,S,T,U,V,W,X,Y,Z}

\docsvlist{A,B,C,D,E,F,G,H,I,J,K,L,M,N,O,P,Q,R,S,T,U,V,W,X,Y,Z}

\docsvlist{A,B,C,D,E,F,G,H,I,J,K,L,M,N,O,P,Q,R,S,T,U,V,W,X,Y,Z}

\docsvlist{A,B,C,D,E,F,G,H,I,J,K,L,M,N,O,P,Q,R,S,T,U,V,W,X,Y,Z}

\docsvlist{A,B,C,D,E,F,G,H,I,J,K,L,M,N,O,P,Q,R,S,T,U,V,W,X,Y,Z}

\renewcommand{\hat}{\widehat}
\newcommand{\R}{\mathbf{R}}
\newcommand{\C}{\mathbf{C}}
\newcommand{\Q}{\mathbf{Q}}
\newcommand{\Z}{\mathbf{Z}}
\newcommand{\N}{\mathbf{N}}
\newcommand{\A}{\mathbf{A}}

\newcommand{\K}{\mathbf{K}}
\renewcommand{\P}{\mathbf{P}}

\newcommand{\m}{\mathfrak{m}}
\newcommand{\geom}{\hspace{.3em}\dot{\bigwedge}\hspace{.3em}}
\newcommand{\G}{\mathbb G}
\newcommand{\T}{\mathbb T}
\newcommand{\D}{\mathbb D}
\newcommand{\HH}{\mathbb H}
\newcommand{\CS}{\mathbb S}
\renewcommand{\div}{\operatorname{div}}
\DeclareMathOperator{\Isom}{Isom}
\DeclareMathOperator{\Div}{Div}
\DeclareMathOperator{\Gal}{Gal}
\DeclareMathOperator{\ord}{ord}
\newcommand{\OO}{\mathcal O}
\DeclareMathOperator{\id}{id}
\DeclareMathOperator{\Tr}{Tr}
\DeclareMathOperator{\Sing}{Sing}
\DeclareMathOperator{\fin}{f}

\DeclareMathOperator{\Supp}{Supp}
\DeclareMathOperator{\dist}{dist}
\renewcommand{\mod}{\text{mod}}
\DeclareMathOperator{\NS}{NS}

\DeclareMathOperator{\Nef}{Nef}
\DeclareMathOperator{\Per}{Per}

\DeclareMathOperator{\spec}{Spec}
\DeclareMathOperator{\Proj}{Proj}
\DeclareMathOperator{\Aut}{Aut}
\DeclareMathOperator{\SL}{SL}
\DeclareMathOperator{\GL}{GL}
\DeclareMathOperator{\PSL}{PSL}
\DeclareMathOperator{\PGL}{PGL}
\DeclareMathOperator{\Out}{Out}
\DeclareMathOperator{\Teich}{Teich}
\DeclareMathOperator{\Mod}{Mod}
\DeclareMathOperator{\DF}{DF}
\DeclareMathOperator{\QF}{QF}
\DeclareMathOperator{\Bers}{Bers}

\newcommand{\DivInf}{\Div_\infty}
\newcommand{\hatDivInf}{\hat{\Div}_\infty}

\DeclareMathOperator{\an}{an}
\newcommand{\dd}{\mathop{\mathrm{{}d}}\mathopen{}}

\newtheorem{thm}{Theorem}[section]
\newtheorem{thm*}{Theorem}
\newtheorem{bigthm}{Theorem}
\newtheorem{prop}[thm]{Proposition}
\newtheorem{cor}[thm]{Corollary}
\newtheorem{lemme}[thm]{Lemma}

\newtheorem{claim}[thm]{Claim}
\theoremstyle{definition}
\newtheorem{dfn}[thm]{Definition}

\newtheorem{rmq}[thm]{Remark}

\AtBeginEnvironment{dfn}{%
  \setlist[enumerate]{label={(\roman*)}}
}

\AtBeginEnvironment{thm}{%
  \setlist[enumerate,1]{label={(\arabic*)}}
}

\AtBeginEnvironment{prop}{%
  \setlist[enumerate,1]{label={(\arabic*)}}
}

\AtBeginEnvironment{lemme}{%
  \setlist[enumerate,1]{label={(\arabic*)}}
}

\begin{document}
\author{Marc Abboud \\ Université de Neuchâtel}
\address{Marc Abboud, Institut de mathématiques, Université de Neuchâtel
\\ Rue Emile-Argand 11 CH-2000 Neuchâtel}
\email{marc.abboud@normalesup.org}
\title{Unlikely intersections problem for automorphisms of Markov surfaces}
\thanks{The author acknowledge support by the Swiss National Science Foundation Grant “Birational transformations of higher dimensional varieties” 200020-214999.}
\maketitle

\begin{abstract}
  We study a problem of unlikely intersections for automorphisms of Markov surfaces of positive entropy. We show for certain
  parameters that two automorphisms with positive entropy share a Zariski dense set of periodic points if and only if
  they share a common iterate. Our proof uses arithmetic equidistribution for adelic line bundles over quasiprojective
  varieties, the theory of laminar currents and quasi-Fuchsian representation theory.
\end{abstract}
\tableofcontents
\section{Introduction}
The Markov surface $\cM_D$ of parameter $D \in \C$ is the affine subvariety of $\C^3$ defined by the equation
\begin{equation}
  x^2 + y^2 + z^2 = xyz + D.
  \label{<+label+>}
\end{equation}

This family of surfaces has been heavily studied as they appear in different areas of mathematics (see
\cite{cantatBersHenonPainleve2009} or \cite[\S2]{rebeloDynamicsGroupsAutomorphisms2022}). We study the dynamics
of polynomial automorphisms of $\cM_D$, that is polynomial transformations of the ambient space $\C^3$ that preserve
$\cM_D$ and are invertible there. A loxodromic
automorphism $f$ is an automorphism with first dynamical degree $\lambda_1 > 1$. Here the first dynamical degree is defined
as follows. Let $X$ be a \emph{compactification} of $\cM_D$, that is a projective surface with an embedding $\cM_D \hookrightarrow
X$ as an open dense subset and let $H$ be an ample Cartier divisor on $X$, then
\begin{equation}
  \lambda_1 (f) := \lim_{n \rightarrow \infty} \left( (f^n)^* H \cdot H \right)^{1/n}
  \label{<+label+>}
\end{equation}
where $\cdot$ is the intersection product on divisors and $(f^n)^*$ is the pull back operator induced by $f^n$ on the
Néron-Severi group of $X$. One shows that this definition does not depend on the choice of $X$ nor $H$. Alternatively,
it is known (see Theorem \ref{ThmAutMarkov}) that the topological entropy of $f$ is equal to $h_{\text{top}}(f) = \log
\lambda_1 (f)$. Therefore, loxodromic automorphisms are the one with positive entropy.

\subsection{Unlikely intersections problem for special parameters}
The main result of this paper is as follows.
\begin{bigthm}\label{BigThmRigidity}
  Let $D = 0$ or $D = 2 - 2 \cos (\frac{\pi}{q})$ with $q \in \Z_{\geq 2}$. If $f$ and $g$ are two loxodromic
  automorphisms of $\cM_D$, then the following statements are equivalent.
  \begin{enumerate}
    \item \label{item:per-zar-dense} $\Per(f) \cap \Per(g)$ is Zariski dense.
    \item \label{item:per-equal} $\Per(f) = \Per (g)$.
    \item \label{item:common-iterates} $\exists N,M \in \Z \setminus \left\{ 0 \right\}, f^N = g^M$.
  \end{enumerate}
\end{bigthm}
Notice that the implication $\eqref{item:common-iterates} \Rightarrow \eqref{item:per-zar-dense}$ follows from
\cite{xiePeriodicPointsBirational2015} which states in
particular that any loxodromic automorphisms has a Zariski dense set of periodic points.
This type of results are called \emph{unlikely intersection} problems in the literature. The equivalence
\eqref{item:per-zar-dense} $\Leftrightarrow$ \eqref{item:per-equal} of
Theorem \ref{BigThmRigidity} was first established by
Baker and DeMarco for endomorphisms of $\P^1$ in \cite{bakerPreperiodicPointsUnlikely2011} and for polarized endomorphisms
of projective varieties in characteristic zero by Yuan and Zhang in \cite{yuanArithmeticHodgeIndex2017} and
\cite{yuanArithmeticHodgeIndex2021}. The
first instance of this result for non-projective varieties is due to Dujardin and
Favre who showed Theorem \ref{BigThmRigidity} for Hénon maps over a number field in
\cite{dujardinDynamicalManinMumford2017}. It is believed that one could lower
the hypothesis of the theorem requiring only $\Per (f) \cap \Per (g)$ to be infinite (see
\cite[Theorem D and Conjecture 3]{dujardinDynamicalManinMumford2017}). In
\cite{cantatFiniteOrbitsLarge2020}, Cantat and Dujardin showed a similar result for subgroups of
automorphisms of projective surfaces: If $\Gamma \subset \Aut(X)$ is a large subgroup of automorphisms of a projective
surface, then $\Gamma$ cannot have a Zariski dense set of finite orbits unless $X$ is a Kummer example, that is the
quotient of an abelian surface by a finite group.

Our proof of Theorem \ref{BigThmRigidity} follows the ideas of the proof in \cite{dujardinDynamicalManinMumford2017}.
However, the existence of the surface $\cM_4$ shows that Theorem \ref{BigThmRigidity} cannot hold for any
parameter $D$ (see \S \ref{SubSecParam4} and \S \ref{SubSecGeneralParameter}). Hence, the proof of
\cite{dujardinDynamicalManinMumford2017} cannot
work for the surfaces $\cM_D$ without a new kind of argument. To finish the proof, we use the interpretation of the
family of surfaces $\cM_D$ given by the Character varieties of the punctured torus and results from hyperbolic geometry
(see \S \ref{SecRepresentationTheory}).

\subsection{Green functions}
Let $f \in \Aut (\cM_D)$ be a loxodromic automorphism. One crucial step in the proof of Theorem \ref{BigThmRigidity} is
the construction of the \emph{Green functions} of $f$ (this is done in \S \ref{SecInvariantAdelicDivisor}). In fact we
need to construct them over $\C$ but also over
non-archimedean fields. We stick to $\C$ for this introduction.

We have that $\cM_D (\C) \subset \C^3$, the Green functions $G^+_f, G^-_f$ are defined as follows
\begin{equation}
  G_f^\pm (p) = \lim_n \frac{1}{\lambda_1(f)^n} \log^+ \left(\left\|  f^{\pm n} (p)  \right\|\right).
  \label{EqDefGreenFunctions}
\end{equation}
We have the following properties (see \cite{bedfordPolynomialDiffeomorphismsC21991}).
\begin{enumerate}
  \item $G_f^+$ is well defined, continuous, nonnegative and plurisubharmonic over $\cM_D(\C)$,
  \item $G_f^+ \circ f = \lambda_1(f) G_f^+$,
  \item $G_f^+ (p) = 0$ if and only if the forward orbit $(f^N(p))_{N \geq 0}$ is bounded.
\end{enumerate}
The function $G_f^-$ satisfies similar properties. We define the Green currents $T_f^+ = dd^c G_f^+$ and $T_f^- =
dd^c G_f^-$. These are positive closed $(1,1)$-currents over $\cM_D(\C)$ and the measure
\begin{equation}
  \mu_f := T_f^+ \wedge T_f^-
\end{equation}
is well defined. It is of finite mass, thus we can suppose
that it is a probability measure. We call it the \emph{equilibrium measure} of $f$. It is $f$-invariant and its support
is called the \emph{Julia set} of $f$. It is contained in the \emph{generalised} Julia set of $f$ which is the compact
$f$-invariant subset $\left\{ G^+_f = G^-_f = 0 \right\}$. The construction is explained for example in
\cite{cantatBersHenonPainleve2009} and \cite{girandDynamicalGreenFunctions2014}.

\subsection{Arithmetic equidistribution}
If $D$ is algebraic, the construction of the Green functions of $f$ can be done over any complete
field $\K_v$ such that $\Q(D) \hookrightarrow \K_v$. This defines an \emph{adelic divisor} over
$\cM_D$ (see \S \ref{SectionAdelicDivisor}) in the sense of \cite{yuanAdelicLineBundles2023}. Yuan and Zhang's
arithmetic equidistribution theorem from \emph{loc.cit} states that the Galois orbits of any generic sequence
of periodic points equidistributes with respect to $\mu_{v}$. This will imply that if $f$ and $g$ share a Zariski dense
set of periodic points, then they have the same equilibrium measure over any complete algebraically closed field.

The last part of the proof uses Representation theory to show that over $\C$ every saddle periodic point is in the
support of $\mu_{f}$ and one can construct a specific saddle fixed point $q(f) \in \cM_D (\C)$ that must have a
non-compact orbit under $g$ if $f$ and $g$ do not share a commmon iterate.

\subsection{The Picard parameter $D = 4$}\label{SubSecParam4}
The parameter $D=4$ is very special because of the following. There is a $2:1$ cover of $\cM_4$ by the
algebraic torus $\G_m^2 = \C^\times \times \C^\times$ given by
\begin{equation}
 \eta: (u,v) \in \G_m^2 \mapsto \left( u + \frac{1}{u}, v + \frac{1}{v}, uv + \frac{1}{uv} \right) \in \cM_4.
  \label{<+label+>}
\end{equation}
If $\sigma$ is the involution on $\G_m^2$ given by $\sigma(u,v) = \left( u^{-1}, v^{-1} \right)$,
then $\cM_4$ is the quotient $\eta$ is the quotient map. The action of the automorphism group is very explicit for the
Picard parameter $D=4$. Indeed, $\GL_2 (\Z)$ acts on $\G_m^2$ by monomial transformations:
\begin{equation}
  \begin{pmatrix}a & b\\
  c & d\end{pmatrix} \cdot (u,v) = \left( u^a v^b, u^c v^d \right).
  \label{<+label+>}
\end{equation}
We have that $\sigma$ corresponds to the action of the matrix $-\id$ and $\GL_2 (\Z) / \langle \sigma \rangle = \PGL_2
(\Z)$ acts on $\cM_4 = \G_m^2 / \langle \sigma
\rangle$. Up to finite index, all the automorphisms of $\cM_4$ are of this form (see Theorem \ref{ThmAutMarkov}). Thus,
all dynamical problems on $\cM_4$ can be lifted to $\G_m^2$. The parameter $D=4$ is the only one where
$\cM_D$ is a finite equivariant quotient of $\G_m^2$.
Theorem \ref{BigThmRigidity} cannot hold for the parameter $D=4$. Indeed, for every monomial transformation $M$ of $\G_m^2$, the periodic points are given
by $(\mu, \omega)$ where $\mu, \omega$ are roots of unity, the Julia set of $M$ is $\CS^1 \times \CS^1$ and the
equilibrium measure is the Lebesgue measure on the Julia set. Thus, when looking at the quotient, we have that every
loxodromic monomial automorphism of $\cM_4$ have the same equilibrium measure, the same Julia set and the same periodic
points.

\subsection{For a general parameter}\label{SubSecGeneralParameter}
 We conjecture that for every $D \neq 4$ Theorem \ref{BigThmRigidity} holds. This is the
affine counterpart to the Kummer example appearing in the result of Cantat and Dujardin. Using a specialization
argument, we show the following result that goes in the direction of this conjecture.

\begin{bigthm}\label{BigThmTranscendentalParameter}
  Let $D \in \C$ be transcendental and let $f,g \in \Aut(\cM_D)$ be loxodromic automorphisms. The
  following assertions are equivalent:
  \begin{enumerate}
    \item $\Per(f) = \Per(g)$.
    \item $\exists N,M \in \Z, f^N = g^M$.
  \end{enumerate}
\end{bigthm}

\subsection{Plan of the paper}
The proof of Theorem \ref{BigThmRigidity} is split into three parts. In the first part, we construct the Green functions,
Green currents and the equilibrium measure of any loxodromic automorphism of $\cM_D$ with $D$ algebraic at both
archimedean and non-archimedean places. We then apply Yuan-Zhang arithmetic equidistribution theorem from
\cite{yuanAdelicLineBundles2023} to show that two loxodromic automorphisms of $\cM_D$ sharing a Zariski dense set of
periodic points must have the same equilibrium measure at every place.

The second part is to apply the method of
Bedford, Lyubich and Smillie in \cite{bedfordPolynomialDiffeomorphismsOfC1993} to show that in $\cM_D (\C)$ every saddle
periodic point of a loxodromic automorphism is in the support of the equilibrium measure. We use the theory of laminar
and strongly approximable currents from \cite{dujardinStructurePropertiesLaminar2005} and apply techniques from
\cite{dujardinLaminarCurrentsBirational2004}.

The third part is to construct a "special" saddle periodic point $q(f)$
which has the following property: the orbit of $q(f)$ under any loxodromic automorphism $g$ that does not share a common
iterate with $f$ is unbounded. To construct $q(f)$ we use the theory of quasi-Fuchsian representation, the simultaneous
uniformization theorem of Bers and Thurston's hyperbolisation theorem for 3-fold fibering over a circle (see
\cite{mcmullenRenormalization3ManifoldsWhich1996}). This third part is where the hypothesis on the parameter $D$ is
used. The specific values of $D$ give an interpretation of $\cM_D$ as representation of the fundamental group of the
punctured torus ($D=0$) or of an orbifold obtained from a genus 1 torus with a singular point of index $q$ ($D = 2 -2
\cos(\frac{\pi}{q})$).

\subsection*{Acknowledgments}
This work was done during my PhD thesis. I would like to thank my PhD advisors Serge Cantat and Junyi Xie for their
guidance. I also thank Juan Souto for answering questions about quasi-Fuchsian representation theory and Xinyi Yuan for
our discussions on adelic divisors. I thank Seung uk Jang for pointing out typos and small mistakes in an earlier version
of this paper. Part of this paper was
written during my visit at Beijing International Center for Mathematical Research which I thank for
its welcome. Finally, I thank the France 2030 framework programme Centre Henri
Lebesgue ANR-11-LABX-0020-01 and
European Research Council (ERCGOAT101053021) for creating an attractive mathematical environment. I also thank the
referee for his/her comments and suggestions.

\section{Adelic divisors over quasiprojective varieties}\label{SectionAdelicDivisor}
\subsection{Weil and Cartier divisors}\label{sec:terminology}
Let $R$ be a Noetherian integral domain, a \emph{variety} over $R$ is a normal flat integral scheme of finite type over
$\spec R$.

Let $X$ be a variety over $R$, a \emph{Cartier divisor} is a global section of the sheaf $\mathcal
K_X^\times / \OO_X^\times$ where $\mathcal K_X$ is the sheaf of rational functions over $X$ and $\OO_X$ the sheaf of
regular functions. Let $\A = \Z, \Q, \R$, an $\A$-Cartier divisor is an element of the form $D = \sum_i a_i D_i$ where
$a_i \in A$ and $D_i$ is a Cartier divisor. An $\A$-\emph{Weil} divisor over $X$ is a formal sum of irreducible
codimension 1 subvarieties with coefficients in $\A$. An $\A$-Weil divisor $D = \sum_i a_i E_i$ is said to be effective
if $a_i \geq 0$ for every $i$. When $\A = \Z$ we will drop the notation $\Z$-Cartier or Weil divisors and just call them
Cartier and Weil divisors.

If $X$ is a variety over $R$, then since $X$ is normal by definition, it is regular in codimension 1 since $R$ is
Noetherian. Thus, any Cartier divisor induces a unique Weil divisor over $X$, see for example \cite[Corollary
II.6.14]{hartshorneAlgebraicGeometry1977}. If $X$ is regular, then the converse is true, every Weil divisor induces
a unique Cartier divisor (\cite[Proposition II.6.11]{hartshorneAlgebraicGeometry1977}). Recall also that since our
varieties are normal, they are regular outside a subset of codimension 2 (\cite[Theorem
I.6.2A]{hartshorneAlgebraicGeometry1977}).
\subsection{Berkovich spaces} For a general reference on Berkovich spaces, we refer to
\cite{berkovichSpectralTheoryAnalytic2012}.
Let $\K_v$ be a complete field with respect to an absolute value $| \cdot |_v$. Let $X_{\K_v} = \spec A$
be an affine $\K_v$-variety, the \emph{Berkovich analytification} $X_{\K_v}^{\an}$ of $X_{\K_v}$ is the set of multiplicative
seminorms over $A$ that extends $| \cdot |_v$. It is a locally ringed space with a contraction map
\begin{equation}
  c : X_{\K_v}^{\an} \rightarrow X_{\K_v}
  \label{<+label+>}
\end{equation}
defined as follows. If $x \in X_{\K_v}^{\an}$, then
\begin{equation}
  c(x) = \ker (x) = \left\{ a \in A : |a|_x = 0 \right\}
  \label{<+label+>}
\end{equation}
where $\left| \cdot \right|_x$ is the seminorm associated to $x$. The topology on $X_{\K_v}^{\an}$ is the coarsest
topology such that for every $\phi \in A$ the evaluation map $x \mapsto \left| \phi \right|_x \in \R$ is continuous.
If $X_{\K_v}$ is a $\K_v$-variety then $X_{\K_v}^{\an}$ is defined by a gluing process using affine charts and we have
the contraction map $c: X_{\K_v}^{\an} \rightarrow X_{\K_v}$. In
particular, if $\phi \in \K_v (X_{\K_v})$ is a rational function with Weil divisor $\div (\phi)$, for $x \in X^{\an}_{\K_v}
\setminus \left( \Supp \div \phi \right)_{\K_v}^{\an}$, we define $ \left| \phi (x) \right| := \left| c^* \phi \right|_x$. If
$X_{\K_v}$ is proper (e.g projective), then $X_{\K_v}^{\an}$ is compact.
If $p \in X_{\K_v}(\K_v)$ is a closed point, then the fiber $c^{-1} (p)$ consists of a single point $| \cdot |_p$ defined by
\begin{equation}
  | a |_p = | a(p) |_v
  \label{<+label+>}
\end{equation}
where $a(p) = a \phantom{.} \mod \phantom{.} p$. This uses the fact that the local field at $p$ is $\K_v$. We thus have
an embedding
\begin{equation}
  \iota_0 = c^{-1} : X_{\K_v} (\K_v) \hookrightarrow X_{\K_v}^{\an}
  \label{<+label+>}
\end{equation}
and we also write $X_{\K_v} (\K_v)$ for its image. With an analog construction, one can show that there is a map
$X_{\K_v}(\overline \K_v) \rightarrow X_{\K_v}^{\an}$ which is not injective: two points have the same image if they are
in the same Galois orbit. We still write $X_{\K_v}(\overline \K_v)$ for its image in $X_{\K_v}^{\an}$.
If $\phi : X_{\K_v} \rightarrow Y_{\K_v}$ is a morphism of varieties, then there exists a unique morphism
\begin{equation}
  \phi^{\an} :
  X_{\K_v}^{\an} \rightarrow Y_{\K_v}^{\an}
\end{equation}
such that the diagram
\begin{equation}
\begin{tikzcd}
  X_{\K_v}^{\an} \ar[r, "\phi^{\an}"] \ar[d] & Y_{\K_v}^{\an} \ar[d] \\
  X_{\K_v} \ar[r, "\phi"] & Y_{\K_v}
  \label{<+label+>}
\end{tikzcd}
\end{equation}
commutes. In particular, if $X_{\K_v} \subset Y_{\K_v}$, then $X_{\K_v}^{\an}$ is isomorphic to $c_Y^{-1}(X_{\K_v})
\subset Y_{\K_v}^{\an}$.

\subsection{Places and restricted analytic spaces}
Let $\K$ be a number field. A \emph{place} $v$ of $\K$ is an equivalence class of absolute values over $\K$. If $v$ is
archimedean then there is an embedding $\sigma : \K \hookrightarrow \C$ such that any absolute value representing $v$ is
of the form $|x| = \left|\sigma(x) \right|^t_\C$ with $0 < t \leq 1$. In that case we will write $|\cdot|_v$ for the
absolute value with $t=1$ and we write $\C_v = \C$. If $v$ is non-archimedean (we also say that $v$ is \emph{finite}),
then it lies over a prime $p$ and we write $|\cdot|_v$ for the
absolute value of $\K$ representing $v$ such that $|p|_v = \frac{1}{p}$. Every finite place $v$ is of the form
\begin{equation}
  v(P) = \# \left( \OO_\K / \m \right)^{-\ord_\m (P)}
  \label{<+label+>}
\end{equation}
for $P \in \K$ where $\m$ is a maximal ideal of $\OO_\K$, the ring of integers of $\K$. We write $\K_v$ be the completion of
$\K$ with respect to $|\cdot|_v$. We write $M(\K)$ for the set of places of $\K$. If $V \subset M(\K)$, we write
$V[\fin]$ for the subset of finite places in $V$ and $V[\infty]$ for the archimedean ones.

If $v$ is archimedean, then define $\OO_v = \C_v$. If $v$ is non-archimedean, we define $\OO_v$ as the ring of elements of
absolute values $\leq 1$ and $\kappa_v$ as the residue field
\begin{equation}
  \kappa_v := \OO_v / \m_v,
  \label{<+label+>}
\end{equation}
where $\m_v$ is the maximal ideal of elements of absolute value $< 1$. In our setting, the non-archimedean absolute
values are always discretly valued and therefore $\OO_v$ is Noetherian.

Let $X$ be a variety over $\K$. For every place $v$ of $\K$, define $X_v := X \times_\K \spec \K_v$. Similarly, if $D$ is
an $\R$-Cartier divisor over $X$ then we denote by $D_v$ its image under the base change. We write $X_v^{\an}$ for the
Berkovich analytification of $X_v$. We also define the global Berkovich analytification of $X$ as
\begin{equation}
  X^{\an} := \bigsqcup_v X_v^{\an}.
  \label{<+label+>}
\end{equation}
Comparing to \cite{yuanAdelicLineBundles2023}, this space is called the \emph{restricted analytic space} of $X$ by Yuan
and Zhang. If $V$ is a set of places, we also define
\begin{equation}
  X_V^{\an} := \bigsqcup_{v \in V} X_v^{\an}.
  \label{<+label+>}
\end{equation}
In particular, we define
\begin{equation}
  X^{\an} [\fin] := \bigsqcup_{v \in M(\K)[\fin]} X^{\an}_v, \quad X^{\an}[\infty] := \bigsqcup_{v \in M(\K)[\infty]}
  X^{\an}_v.
  \label{<+label+>}
\end{equation}
If $\sX$ is a variety over $\OO_\K$, we write $\sX_v$ for the base change
\begin{equation}
  \sX_v = \sX \times_{\OO_\K} \spec \OO_v.
  \label{<+label+>}
\end{equation}
Similarly, if $\sD$ is an $\R$-Cartier divisor over $\sX$, we denote by $\sD_v$ its image under the base change.

\subsection{Adelic divisors over a projective variety}
Recall the notions of $\A$-Cartier and Weil divisors from \S \ref{sec:terminology}.
\subsection*{Models, horizontal and vertical divisors} If $D = \sum_i a_i D_i$ is an $\R$-Cartier divisor on $X$, a \emph{model} of
$(X,D)$ is the data of $(\sX, \sD)$ where $\sX$ is a normal projective variety over $\OO_\K$ with generic fiber $X$ and
$\sD = \sum_i a_i \sD_i$ is an $\R$-Cartier divisor on $\sX$ such that ${\sD_i}_{|X} = D_i$.

There are two types of $\R$-\emph{Weil} divisors on a projective variety $\sX$
over $\OO_\K$: \emph{horizontal divisors} whose irreducible components are the closure of prime divisors over the
generic fiber and \emph{vertical divisors} whose irreducible components do not intersect the generic fiber. Every
$\R$-Weil divisor $\sW$ over $\sX$ can be uniquely split into a sum $\sW = W_{\text{hor}} + \sW_{\text{vert}}$ of a horizontal divisor and a
vertical one. In particular, $W_{\text{hor}} := \sW_{|X}$ is the restriction of $\sW$ to the generic fiber.

Let $V \subset \OO_\K$ we say that $\sD$ is \emph{horizontal} over $V$ if the Weil divisor over $\sX_V$ defined by
$\sD_V$ is horizontal. In that case we will write $\sD_V = D_V$ for the Cartier divisor $\sD_V$ over $\sX_V$.

\subsection*{Green functions} A \emph{Green function} of an $\R$-Cartier divisor $D$ is a continuous function $g :
X^{\an} \setminus (\Supp D)^{\an} \rightarrow \R$ such that for any point $p \in (\Supp D)^{\an}$, if $z_i$ is a local
equation of $D_i$ at $p$ then the function
\begin{equation}
  g + \sum_i a_i \log |z_i|
  \label{<+label+>}
\end{equation}
extends locally to a continuous function at $p$. For any place $v \in M (\K)$, we write $g_v$ for $g_{|X^{\an}_v}$,
$g[\fin] = g_{|X^{\an}[\fin]}$ and $g[\infty] = g_{X^{\an}[\infty]}$. For the archimedean places $v$ we add
the extra condition that $g_v$ must be invariant under complex conjugation. That is if $\overline v$ is the conjugate
place of $v$, then $g_{\overline v} (z) = g_v (\overline z)$ for $z \in X(\C)$. Notice that compared to
\cite{moriwakiAdelicDivisorsArithmetic2016}, our definition of Green functions differs by a factor 2.

If $(\sX, \sD)$ is a model of $(X,D)$, then for every finite place $v$, $(\sX_v, \sD_v)$ induces a Green function of
$D_v$ over $X^{\an}_v$ as follows. We have the anticontinuous reduction map (see \cite{berkovichSpectralTheoryAnalytic2012})
\begin{equation}
  r_{\sX_v} : X^{\an}_v \rightarrow (\sX_v)_{\kappa_v}.
  \label{<+label+>}
\end{equation}
Let $x \in X^{\an}_v \setminus (\Supp D)_v^{\an}$, and let $z_i$ be a local equation of $\sD_{i,v}$ at $r_{\sX_v}
(x)$, we define
\begin{equation}
  g_{(\sX_v, \sD_v)} (x) := - \sum_i a_i \log |z_i(x)|.
  \label{<+label+>}
\end{equation}
It does not depend on the choice of the local equations $z_i$ because an invertible regular function $\phi$ at
$r_{\sX_v}(x)$ satisfies $\left| \phi(x) \right|= 1$ (recall that here $v$ is non-archimedean).

\subsection*{Divisorial points}\label{subsec:divisorial-points}
Let $v$ be a finite place of $\K$. Let $\sX_v$ be a projective variety over $\OO_v$ with $X_v = \sX_v \times_{\OO_v} \K_v$. The
special fiber $\sX_v \times_{\OO_v} \kappa(v)$ has finitely many irreducible components $E_1, \dots, E_r$. There exists a unique
point $x_i \in X_v^{\an}$ such that $r_{\sX_v} (x_i)$ is the generic point of $E_i$. Any point in $X_v^{\an}$ that
arises like this is called a \emph{divisorial} point.
\begin{prop}[{\cite[2.4.9]{mustataWeightFunctionsNonArchimedean2015}}]\label{prop:density-divisorial-points}
  The set of divisorial points is dense in $X^{\an}_v$.
\end{prop}
If $v$ is a finite place of $\K$, then it is discretely valued and $\OO_v$ is Noetherian. In that case, the local ring
of the generic point of an irreducible component $E_i$ is a Noetherian regular local ring of dimension 1, hence a
discrete valuation ring. Let $\ord_{E_i}$ be the associated valuation, then the point $x_i$ corresponds the seminorm
given by $e^{-\ord_{E_i}}$. The proof of Proposition \ref{prop:density-divisorial-points} implies the following.
\begin{cor}\label{cor:dense-sequence}
  Let $x \in X^{\an}_v$ and let $\sX_v$ be a model of $X_v$ over $\OO_v$. Define the following sequence of models:
  $\sX_0 = \sX_v$ and $\sX_{n+1}$ is the normalised blow up of $\sX_n$ along $\overline{r_{\sX_n} (x)}$ and
  $E_{n+1}$ is the exceptional divisor (the one that is sent dominantly to $\overline{r_{\sX_n} (x)}$). Let $x_n$ be the
  divisorial point associated to $E_n$, then the sequence $(x_n)$ converges to $x$.
\end{cor}

\subsection*{Adelic divisors} An \emph{adelic} divisor $\overline D$ on $X$ is the data of $\overline D = (D, g)$ where
$D$ is an $\R$-divisor on $X$
and $g$ is a continuous
Green function of $D$ over $X^{\an}$ such that
\begin{enumerate}[label=(\roman*)]
    \item There exists an open subset $V \subset \spec \OO_\K$ and a
      model $(\sX_V, \sD_V)$ of $(X,D)$ over $V$ such that for every finite place $v $ lying over $V$,
\begin{equation}
  g_v = g_{(\sX_v, \sD_v)}.
  \label{<+label+>}
\end{equation}
\item For any finite place not lying over $V$ and all the archimedean ones, the Green function $g_v$ is the uniform
  limit of model Green functions of $D$ over $X^{\an}_v$.

  \end{enumerate}
A \emph{model} adelic divisor on $X$ is an adelic divisor $(D,g)$ on $X$ such that $g[\fin]$ is induced by a model
$(\sX, \sD)$. It is \emph{vertical} if $\sD$ is vertical, in particular $D = 0$.

\begin{dfn}\label{dfn:positivity-adelic-divisors-projective-variety}
  Following \cite{zhangSmallPointsAdelic1993},
  \begin{enumerate}
    \item A model adelic divisor $(\sX, \sD)$ is \emph{semipositive} if $g[\infty]$ is plurisubharmonic
      (psh) and $\sD$ is nef over $\sX$.
    \item An adelic divisor $\overline D = (D, g)$ is \emph{semipositive} if $g$ is the uniform
  limit of semipositive models of $(X, D)$.
  \item An adelic divisor is \emph{integrable} if it is the difference of two
semipositive adelic divisors.
\item An adelic divisor $\overline D = (D,g)$ is \emph{effective} if $g \geq 0$. In particular this implies that
$D$ is an effective divisor.
 \end{enumerate}
\end{dfn}
We write $\overline D \geq \overline D'$ if $\overline D - \overline D'$ is effective. We
say that $\overline D$ is \emph{strictly} effective if $\overline D \geq 0$ and $g[\infty] > 0$.

\begin{lemme}\label{lemme:effective-semipositive-model}
Every ample effective Cartier divisor $H$ on $X$ admits a strictly effective semipositive model.
\end{lemme}
\begin{proof}
 Let $\iota: X \hookrightarrow \P^N_\K$ be a projective embedding such that $H = \left\{ T_0 = 0 \right\} \cap \iota
 (X)$ where $T_0, \cdots, T_N$ are the homogeneous coordinates of $\P^N_\K$. We can find a model $\sX$ such that $\iota$
 extends to a regular map
\begin{equation}
  \overline \iota : \sX \rightarrow \P^N_{\OO_\K}
  \label{<+label+>}
\end{equation}
then the Cartier divisor $\sD = \overline \iota^* \left\{ T_0 = 0 \right\}$ is a model of $H$ and a nef Cartier divisor over $\sX$
(Here $T_0$ is a homogeneous coordinate of $\P^N_{\OO_\K}$). For every archimedean place $v$, $g_v$ is given by
\begin{equation}
  g_v (p) = 1 + \log^+ \max \left( | t_1 (p)|, \cdots, | t_N (p)|  \right)
  \label{<+label+>}
\end{equation}
where $t_i := \frac{T_i}{T_0}$. This is (one plus) the pull back of the \emph{Weil} metric over $\P^N$ (see e.g
\cite{chambert-loirHeightsMeasuresAnalytic2011}).
\end{proof}

\begin{prop}[Lemma 3.3.3 of \cite{yuanAdelicLineBundles2023}]\label{prop:effectivity-model-divisors}
  If $\overline \sD = (D,g)$ is a model adelic divisor, then $\overline \sD$ is effective
  if and only if $g[\infty] \geq 0$ and $\sD$ is effective as an $\R$-Weil divisor.
\end{prop}
\begin{proof}
  This is shown in \cite{yuanAdelicLineBundles2023} only when $\sD$ is a $\Q$-Cartier divisor. Let $\sX$ be a model of
  $X$ over $\OO_\K$ where $\sD$ is defined. First of all if
  $g[\infty] \geq 0$, then it is clear that $D = \sD_{\text{hor}}$ is an effective $\R$-Weil divisor. Indeed, if $D = \sum_i
  a_i E_i$ with $a_i \in \R$ and $E_i$ are distinct irreducible codimension 1 subvarieties of $X$, then at a general
  point of $E_i$ we have that $g[\infty] = - a_i \log \left| z_i \right| + h$ where $z_i$ is a local equation of $E_i$
  (since our varieties are normal, they are regular outside a set of codimension 2 so $E_i$ defines a Cartier divisor at
  a general point) and $h$ is continous. Since $g[\infty] \geq 0 $ this shows that $a_i \geq 0$.

  Assume that $\overline \sD$ is effective. We have that $g[\infty] \geq 0$ and that $\sD_{\text{hor}}$ is an effective
  $\R$-Weil divisor. We now show that $\sD_{\text{vert}}$ is an effective $\R$-Weil divisor. Let $E \subset \sX$ be a vertical irreducible
  component of codimension 1 and let $v \in \spec \OO_\K$ be the place such that $E$ lies over $v$. Then, by \S
  \ref{subsec:divisorial-points}, there exists $x_E \in X_v^{\an}$ such that $r_{\sX_v}(x_E) = \eta_E$ the generic point
  of $E$ and $x_E$ is given by the seminorm $e^{-\ord_E}$. Write $\sD = \sum_i a_i \sD_i$ where $a_i \in \R$
  and $\sD_i$ are Cartier divisors with local equations $z_i$ at $\eta_E$. We have,
  \begin{equation}
    0 \leq g(x_E) = - \sum_i a_i \log \left| z_i (x_E) \right| = \sum_i a_i \ord_E (\sD_i) = \ord_E (\sD).
    \label{eq:<+label+>}
  \end{equation}
  This shows one implication.

  Conversely, if $g[\infty] \geq 0$ and $\sD$ is an effective $\R$-Weil divisor, then this implies that $\sD_{\text{vert}}$ is
  a vertical effective $\R$-Weil divisor. This implies that for any finite place $v$ and for any divisorial point
  $x_E \in X^{\an}_v$ we have $g(x_E) = \ord_E (\sD) \geq 0$. Since the divisorial points are dense by Proposition
  \ref{prop:density-divisorial-points} we get that $g[\fin] \geq 0$.
\end{proof}
A corollary of the proof is the following.
\begin{cor}\label{cor:Green-function-zero-if-no-support}
  Let $\overline \sD$ be a model adelic divisor and $\sX$ a model where $\sD$ is defined. Let $v$ be a finite place of
  $\K$ and let $x \in X^{\an}_v$. If $r_{\sX_v}(x)$ is not in the support of the $\R$-Weil divisor $\sD_v$, then
  $g_{(\sX, \sD)} (x) = 0$.
\end{cor}
\begin{proof}
  Let $(x_n)$ be the sequence associated to $x$ defined in Corollary \ref{cor:dense-sequence}. Since $r_{\sX}(x)$ is not
  in the support of $\sD$ we have that for every $n, g_{(\sX, \sD)} (x_n) = 0$ since for every vertical irreducible
  divisor $E_n$, $\sD$ has no support over $E_n$. Since $x_n \rightarrow x$ we have the result.
\end{proof}
\begin{rmq}\label{rmq:max-not-possible}
  In \cite{dujardinDynamicalManinMumford2017}, the authors used the theory of adelic divisors over $\P^2$. Their
  approach was two construct Green functions $G^+, G^-$ which are Green functions of the line at infinity except at one
  point and to use the following trick: The Green function $G = \max(G^+, G^-)$ is a Green function of the line at
  infinity over $\P^2$. This works because $\P^2 \setminus \A^2$ has only one irreducible component.
In our setting, using adelic divisors over a projective variety will not be enough because the compactifications of $\cM_D$ have
several components at infinity. In general if $\overline D_1, \overline D_2$ are adelic divisors with Green functions
$g_1, g_2$ respectively, then $\max(g_1, g_2)$ is not the Green function of any $\R$-Cartier divisor. Indeed, suppose for
example that $D_1 = E$ and $D_2 = F$ are irreducible divisors that intersect simply and normally. That means that for
any point $p \in E \cap F$, $E \cup F$ is locally given by the equation $xy = 0$ where $x$ is a local equation of $E$
and $y$ is the local equation of $F$. Then, $g = \max (g_1, g_2)$ should be the Green function of a Cartier divisor $D$
supported on
$E \cup F$. But if we're approaching $p$ by staying very close to the axis $x = 0$, then we get that $g \simeq \log
\left| x \right| $ at $p$ so $D = E$. But the same argument with respect to the axis $y = 0$ shows that $D = F$ and this
is a contradiction. This explains why we have to use the theory of Yuan and Zhang over quasiprojective varieties.
\end{rmq}

\subsection{Over quasiprojective varieties}
The main reference for this section is \cite{yuanAdelicLineBundles2023}.
Let $U$ be a normal quasiprojective variety over a number field $\K$. A \emph{quasiprojective model} $\sU$ of $U$ is a
quasiprojective variety $\sU$ over $\spec \OO_\K$ with generic fiber isomorphic to $U$. A \emph{projective model} of $\sU$ is a
projective variety $\sX$ over $\spec \OO_\K$ with an open embedding $\sU \hookrightarrow \sX$.

A model adelic divisor on $\sU$ is a model adelic divisor defined over a projective model of $\sU$. If $\overline
\sD$ is a model adelic divisor on some projective model $\sX$ of $\sU$. We say that $\overline \sD$ is \emph{supported at
infinity} if $(\sD_{\text{hor}})_{|U} = 0$.  We write $\hat
\DivInf (\sX, \sU)$ for the set of model adelic divisor \emph{supported at infinity} induced by a fixed projective model $\sX$ of
$\sU$. Since the system of projective models of $U$ is a projective system, we can define the direct limit
\begin{equation}
  \hatDivInf (\sU)_\mod := \varinjlim_{\sX} \hatDivInf(\sX, \sU).
  \label{<+label+>}
\end{equation}
Our definition differs from \cite{yuanAdelicLineBundles2023} because we only take divisors supported outside $U$. This
makes sense for our dynamical setting.

A \emph{boundary divisor} on $\sU$ is a strictly effective model adelic divisor $\overline \sD_0 = (\sD_0, g_0)$ over a
projective model $\sX_0$ of $\sU$ such that $\Supp \sD_0 = \sX_0 \setminus \sU$. It defines a norm on
$\hatDivInf(\sU)_\mod$ given by
\begin{equation}
  \parallel \overline \sD \parallel = \inf \left\{ \epsilon > 0 : - \epsilon \overline \sD_0 \leq \overline \sD \leq \epsilon
  \overline \sD_0 \right\}.
  \label{<+label+>}
\end{equation}
An adelic divisor $\overline D$ on $\sU$ is an element of the compactification of $\hatDivInf(\sU)_\mod$ with respect to this
norm. More
precisely, an adelic divisor on $\sU$ is a sequence of model adelic divisors $(\sX_i, \overline \sD_i)$ such that there
exists a sequence $\epsilon_i > 0, \epsilon_i \rightarrow 0$ such that
\begin{equation}
  \forall j \geq i, \quad - \epsilon_i \overline \sD_0 \leq \overline \sD_j - \overline \sD_i \leq \epsilon_i \overline
  \sD_0.
  \label{<+label+>}
\end{equation}
If we denote by $g_i$ the Green function of the model adelic divisor $\osD_i$, then this is equivalent to asking that
\begin{equation}
  -\epsilon_i g_0 \leq g_j - g_i \leq \epsilon_i g_0.
  \label{<+label+>}
\end{equation}
In particular, $(g_i)$ converges uniformly locally to a continuous function $g$ over $U^{\an}$. We call it the
\emph{Green} function of $\overline D$.
We write $\hatDivInf(\sU)$ for the compactification of $\hatDivInf(\sU)$ with respect to this norm.
An adelic divisor on $U$ is an element of
\begin{equation}
  \hatDivInf(U / \OO_\K) := \varinjlim_{\sU} \hatDivInf(\sU)
  \label{<+label+>}
\end{equation}

\begin{rmq}
  In \cite{yuanAdelicLineBundles2023} Yuan and Zhang use only $\Q$-model divisors for the definition of $\DivInf(\sX, \sU)$
  and $\hatDivInf(\sU)$ whereas here we also allow $\R$-model adelic divisors. If one wants to use adelic line bundles
  instead of adelic divisors and in particular the global section of line
  bundles then it make sense to use only $\Q$-line bundles in the limit process defining adelic line bundles with the
  boundary topology but we do not need that here. Anyway, for adelic divisors, using $\R$-model divisors provides the
  same final space thanks to the following lemma.
\end{rmq}

\begin{lemme}
  Let $\osD$ be a model $\R$-adelic divisor supported at infinity and let $\osD_0$ be a boundary divisor. Then for any
  $\epsilon > 0$, there exists a $\Q$-model divisor $\osD_\epsilon$ such that
  \begin{equation}
    - \epsilon \osD_0 \leq \osD - \osD_\epsilon \leq \epsilon \osD_0
    \label{<+label+>}
  \end{equation}
\end{lemme}
\begin{proof}
  Let $\sX$ be a model over $\spec \OO_\K$ such that $\sD, \sD_0$ are defined. Write $\sD = \sum_i a_i \sD_i$ where
  $a_i \in \R$ and $\sD_i$ are Cartier divisors over $\sX$. By Proposition \ref{prop:effectivity-model-divisors}, if
  $a_{i,n}$ is a sequence of rational numbers converging to $a_i$, then the Green function of the model divisor $\sD_n
  := \sum_i a_{i,n} \sD_i$ converges towards the Green function of $\sD$ with respect to the boundary topology over
  every non-archimedean place. We just need to construct the Green function of $\sD_n$ at the archimedean places. Fix an
  archimedean place $v$, and let $g : U(\C) \rightarrow \R$ be the Green function of $\osD$
  over $v$. Let $D_1, \cdots, D_r$ be the restrictions of $\sD_1, \dots, \sD_r$ to $X$ and let $g_i$ be a Green function
  of $D_i$ over $U(\C)$ then $g - \sum_i a_i g_i = h$ is a continuous bounded function over
  $U(\C)$. We can suppose without loss of generality that $a_1 \neq 0$ and we replace $g_1$ by $g_1 + \frac{1}{a_1} h$ such that
  \begin{equation}
    g = \sum_i a_i g_i.
    \label{<+label+>}
  \end{equation}
  Now, let $A > 0$ such that for all $i, \sup_{U(\C)} \left|\frac{g_i}{g_0}\right| \leq A$ and let $a_{i,n}$ be a sequence of
  rational numbers converging towards $a_i$, then
  \begin{equation}
    \left|\frac{g - \sum_i a_{i,n} g_i}{g_0}\right| \leq r A \max_i \left( \left|a_i - a_{i,n}\right| \right)
    \label{<+label+>}
  \end{equation}
  and we have the result.
\end{proof}

\begin{dfn}
An adelic divisor $\overline D$ over $U$ is
\begin{enumerate}
  \item \emph{strongly nef} if for the Cauchy sequence $(\overline \sD_i)$ defining
it we can take for every $\overline \sD_i$ a semipositive model adelic divisor.
\item \emph{nef} if there exists a strongly nef adelic divisor $\overline A$ such that for all $m \geq 1, \overline D +
  m \overline A$ is strongly nef.
\item \emph{integrable} if it is the difference of two strongly nef adelic divisors.
\end{enumerate}
\end{dfn}

If $\overline D$ is an adelic divisor over $U$, then $\overline D$ has an associated \emph{height function}
\begin{equation}
  h_{\overline D} : U(\overline \K) \rightarrow \R
  \label{<+label+>}
\end{equation}
which is computed as follows if $D_{|U} = 0$ which is our use case in this paper:
\begin{equation}
  \forall p \in U (\overline K), \quad h_{\overline D} (p) = \frac{1}{\deg p} \sum_{v \in
  M(\K)} \sum_{q \in \Gal(\overline \K / \K)\cdot p} n_v g_{\overline D, v} (q)
  \label{<+label+>}
\end{equation}
where $n_v = [\K_v : \Q_v]$. Moreover, for any closed $\overline \K$-subvariety $Z$ of $U$, we define the height of $Z$
to be
\begin{equation}
  h_{\overline D} (Z) := \frac{\overline D_{|Z}^{\dim Z + 1}}{(1 + \dim Z)D_{|Z}^{\dim Z}}
  \label{<+label+>}
\end{equation}
where $\overline D_{|Z}^{\dim Z + 1}$ represents the intersection number of adelic divisors. See
\cite{yuanAdelicLineBundles2023} for more details.

\begin{prop}[{\cite[\S2.5.5]{yuanAdelicLineBundles2023}}]\label{PropFunctoriality}
  If $f : X \rightarrow Y$ is a morphism between quasiprojective varieties over $\K$, then there is a pullback operator
  \begin{equation}
    f^* : \hatDivInf( Y / \OO_\K) \rightarrow \hatDivInf (X/ \OO_\K)
    \label{<+label+>}
  \end{equation}
  that preserves model, strongly nef, nef and integrable adelic divisors. If $g$ is the Green function of $\overline D
  \in \hatDivInf(Y / \OO_\K)$, then the Green function of $f^* \overline D$ is $g \circ f^{\an}$.
\end{prop}

\section{Picard-Manin space at infinity}
\subsection{Compactifications}
A compactification of $\cM_D$ is a projective surface $X$ with an open embedding $\iota_X : \cM_D \hookrightarrow X$
such that $X$ is smooth in a neighbourhood of $X \setminus \cM_D$. We call
$X \setminus \iota_X (\cM_D)$ the \emph{boundary} of $\cM_D$ in $X$. By \cite[Proposition
1]{goodmanAffineOpenSubsets1969}, it is a connected curve. We will also refer to it at the part "at infinity" in $X$.
For any compactification $X$ of $\cM_D$ we define $\DivInf (X)_\A = \oplus \A E_i$ where $\A = \Z, \Q, \R$ and $X
\setminus \cM_D = \bigcup E_i$, the space of $\A$-divisors at infinity. Notice that since $X$ is smooth at infinity
every element of $\DivInf(X)_\A$ is both an $\A$-Cartier and Weil divisor. For any two
compactifications $X,Y$ we have a
birational map $\pi_{XY} = \iota_Y \circ \iota_X^{-1} : X \dashrightarrow Y$. If this map is regular, we say that
$\pi_{XY}$ is a \emph{morphism of compactifications} and that $X$ is \emph{above} $Y$. For any compactification
$X,Y$ there exists a compactification $Z$ above $X$ and $Y$. Indeed, take $Z$ to be a resolution of indeterminacies of
$\pi_{XY} : X \dashrightarrow Y$. A morphism of compactifications defines a pullback and a pushforward operator
$\pi_{XY}^*, (\pi_{XY})_*$ on divisors and Néron-Severi classes. We have the projection formula,
\begin{equation}
  \forall \alpha \in \NS (X), \beta \in \NS (Y), \quad \alpha \cdot \pi_{XY}^* \beta = (\pi_{XY})_* \alpha \cdot \beta.
  \label{<+label+>}
\end{equation}

Let $\overline \cM_D \subset \P^3$ be the closure of $\cM_D$ in $\P^3$. We have that $\overline \cM_D \setminus \cM_D$
is a triangle of lines all of self intersection $-1$.  The matrix of the intersection form on
$\DivInf(\overline \cM_D)$ is
\begin{equation}
  \begin{pmatrix}
    -1 & 1 & 1 \\
    1 & -1 & 1 \\
    1 & 1 & -1
  \end{pmatrix}
  \label{<+label+>}
\end{equation}
Therefore, the intersection form is non-degenerate over $\DivInf(\overline \cM_D)_\A$ and we have the embedding
\begin{equation}
  \DivInf(\overline \cM_D)_\A \hookrightarrow \NS (\overline \cM_D)_\A.
  \label{<+label+>}
\end{equation}
And this holds for every compactification $X$ of $\cM_D$.

\subsection{Weil and Cartier classes}
If $\pi_{YX} : Y \rightarrow X$ are two compactifications of $\cM_D$ then we have the embedding defined by the
pullback operator
\begin{equation}
  \pi_{YX}^* : \DivInf(X)_\A \hookrightarrow \DivInf(Y)_\A.
  \label{<+label+>}
\end{equation}
We define the space of Cartier divisors at infinity of $\cM_D$ to be the direct limit
\begin{equation}
  \cC (\cM_D) := \varinjlim_{X} \DivInf(X)_\R.
  \label{<+label+>}
\end{equation}
In the same way we define the space of Cartier classes of $\cM_D$
\begin{equation}
  \cZ (\cM_D) := \varinjlim_X \NS(X)_\R.
  \label{<+label+>}
\end{equation}
An element of $\cZ (\cM_D)$ is an equivalence class of pairs $(X, \alpha)$ where $X$ is a compactification of $\cM_D$ and
$\alpha \in \NS (X)_\R$ such that $(X, \alpha) \simeq (Y, \beta)$ if and only if there exists a compactification $Z$ above
$X,Y$ such that $\pi_{ZX}^* \alpha = \pi^*_{ZY} \beta$. We say that $\alpha \in \cZ (\cM_D)$ is \emph{defined} in $X$ if it
is represented by $(X, \alpha)$.
We have a natural embedding $\cC (\cM_D) \hookrightarrow \cZ(\cM_D)$, we still write $\cC (\cM_D)$ for its image in $\cZ
(\cM_D)$. We also define the space of Weil classes
\begin{equation}
  \hat \cZ (\cM_D) := \varprojlim_{X} \NS(X)_\R
  \label{<+label+>}
\end{equation}
where the compatibility morphisms are given by the pushforward morphisms $(\pi_{YX})_* : \NS (Y) \rightarrow \NS(X)$ for
a morphism of
compactifications $\pi_{YX}: Y \rightarrow X$.
An element of this inverse limit is a family $\alpha = (\alpha_X)_X$ such that if $X,Y$ are two compactifications of $\cM_D$
with $Y$ above $X$, then $(\pi_{YX})_* \alpha_Y = \alpha_X$. We call $\alpha_X$ the \emph{incarnation} of
$\alpha$ in $X$. We have a natural embedding $\cZ (\cM_D) \hookrightarrow \hat \cZ (\cM_D)$. We also define the space of Weil
divisors at infinity
\begin{equation}
  \cW (\cM_D) := \varprojlim_{X} \DivInf(X)_\R
  \label{<+label+>}
\end{equation}
and we have the commutating diagram
\begin{equation}
  \begin{tikzcd}
    \cC(\cM_D) \ar[r, hook] \ar[d, hook] & \cZ (\cM_D) \ar[d, hook] \\
    \cW(\cM_D) \ar[r, hook] & \hat \cZ(\cM_D)
  \end{tikzcd}
  \label{<+label+>}
\end{equation}
Thanks to the projection formula, the intersection form defines a perfect pairing
\begin{equation}
  \cZ (\cM_D) \times \hat \cZ (\cM_D) \rightarrow \R
  \label{<+label+>}
\end{equation}
defined as follows. If $\alpha \in \cZ (\cM_D)$ is defined in $X$ and $\beta \in \hat \cZ (\cM_D)$, then
\begin{equation}
  \alpha \cdot \beta = \alpha_X \cdot \beta_X
  \label{<+label+>}
\end{equation}
An element $\alpha \in \cW(\cM_D)$ is \emph{effective} if for every compactification $X$, $\alpha_X$ is an effective
divisor. We write $\alpha \geq \beta$ if $\alpha - \beta$ is effective. An element $\beta \in \hat \cZ (\cM_D)$ is
\emph{nef} if for every compactification $X, \beta_X$ is nef.

\subsection{The Picard-Manin space of $\cM_D$}
We provide $\hat \cZ (\cM_D)$ with the topology of the inverse limit, we call it the weak topology, $\cZ(\cM_D)$ is
dense in $\hat \cZ (\cM_D)$ for this
topology. Analogously, $\cC (\cM_D)$ is dense in $\cW(\cM_D)$.

We define $\cD_\infty$ for the set of prime divisors at infinity. An element of $\cD_\infty$ is an equivalence class of
pairs $(X,E)$ where $X$ is a compactification of $\cM_D$ and $E$ is a prime divisor at infinity. Two pairs $(X,E), (Y, E')$
are equivalent if the birational map $\pi_{XY}$ sends $E$ to $E'$. We will just write $E \in \cD_\infty$ instead of
$(X,E)$. We define the function $\ord_E : \cW (\cM_D) \rightarrow \R$ as follows. Let $\alpha \in \cW(\cM_D)$, if $X$ is any
compactification where $E$ is defined (in particular $(X,E)$ represents $E \in \cD_\infty$), then $\alpha_X$ is of the form
\begin{equation}
  \alpha_X = a_E E + \sum_{F \neq E} a_F F
  \label{<+label+>}
\end{equation}
and we set $\ord_E (\alpha_X) = a_E$. This does not depend on the choice of $(X,E)$.
\begin{lemme}[\cite{boucksomDegreeGrowthMeromorphic2008} Lemma 1.5]\label{LemmeWeakTopology}
  The map
  \begin{equation}
    \alpha \in \cW (\cM_D) \mapsto (\ord_E (\alpha))_{E \in \cD_\infty} \in \R^{\cD_\infty}
    \label{<+label+>}
  \end{equation}
  is a homeomorphism for the product topology.
\end{lemme}
We can define a stronger topology on $\cZ (\cM_D)$ as follows. The intersection product defines a non-degenerate bilinear form
\begin{equation}
  \cZ (\cM_D) \times \cZ (\cM_D) \rightarrow \R
  \label{<+label+>}
\end{equation}
of signature $(1, \infty)$ by the Hodge Index Theorem.
 Take an ample class $\omega$ on some compactification $X$ of $\cM_D$ such that $\omega^2 =1$. Every Cartier class $\alpha \in
 \cZ (\cM_D)$ can be decomposed with respect to $\omega$ and $\omega^\perp$
\begin{equation}
  \alpha = (\alpha \cdot \omega) \omega + (\alpha - (\alpha \cdot \omega)\omega).
  \label{<+label+>}
\end{equation}
By the Hodge index theorem, the intersection form is negative definite on $\omega^\perp$ and we define the norm
\begin{equation}
  \left\|  \alpha  \right\|_\omega^2 = (\alpha \cdot \omega)^2 - (\alpha - (\alpha \cdot \omega)\omega)^2.
  \label{<+label+>}
\end{equation}
This defines a topology on $\cZ (\cM_D)$ which is independent of the choice of $\omega$. We call it the strong topology. We
define $\overline \cZ (\cM_D)$ to be the compactification of $\cZ (\cM_D)$ with respect to this topology. As this topology is
stronger than the weak topology, $\overline \cZ (\cM_D)$ is a subspace of $\hat \cZ (\cM_D)$.
This is a Hilbert space and
the intersection product extends to a continuous non-degenerate bilinear form
\begin{equation}
  \overline \cZ (\cM_D) \times \overline \cZ (\cM_D) \rightarrow \R.
  \label{<+label+>}
\end{equation}
We call $\overline \cZ (\cM_D)$ the \emph{Picard-Manin} space of $\cM_D$. We also write $\overline \cC (\cM_D)$ for the closure
of $\cC (\cM_D)$ for the strong topology.
We have in particular that every nef class in $\hat \cZ (\cM_D)$ belongs to $\overline \cZ (\cM_D)$ (see
\cite{boucksomDegreeGrowthMeromorphic2008} Proposition 1.4).

\begin{rmq}
  In \cite{boucksomDegreeGrowthMeromorphic2008} or \cite{cantatNormalSubgroupsCremona2013}, the Picard-Manin space is
  defined by allowing blow up with arbitrary centers not only at infinity. Since we study dynamics of automorphism of
  $\cM_D$ the indeterminacy points are only at infinity. This justifies our restricted definition of the Picard-Manin
  space. A similar construction is used in \cite{favreDynamicalCompactificationsMathbf2011} for the affine plane.
\end{rmq}

\subsection{Spectral property of the dynamical degree}
If $f \in \Aut (\cM_D)$ we define the operator $f^*$ on $\cZ(\cM_D)$ as follows. Let $\alpha \in \cZ(\cM_D)$
defined in a compactification $X$. Let $Y$ be a compactification of $\cM_D$ such that the lift $F: Y \rightarrow X$ of $f$ is
regular. We define $f^* \alpha$ as the Cartier class defined by $F^* \alpha$. This does not depend on the choice of $X$
or $Y$. We write $f_*$ for $(f^{-1})^*$. If $X$ is a compactification of $\cM_D$, we write $f_X^* : \DivInf(\cM_D) \rightarrow
\DivInf(\cM_D)$ for the following operator:
\begin{equation}
  f_X^* (D) = (f^* D)_X
  \label{<+label+>}
\end{equation}
where we consider the class of $D$ and $f^*D$ in $\cC (\cM_D)$. We also define the operator $f_X^* : \NS(X) \rightarrow
\NS(X)$ in a similar way.
\begin{prop}[Proposition 2.3 and Theorem 3.2 of \cite{boucksomDegreeGrowthMeromorphic2008}]
  The operator $f^*$ extends to a continuous bounded operator $f^* : \overline \cZ (\cM_D) \rightarrow
  \overline \cZ (\cM_D)$ that satisfy the following conditions.
  \begin{enumerate}
    \item $f^* \alpha \cdot \beta = \alpha \cdot f_* \beta$.
    \item $f^* \alpha \cdot f^* \beta = \alpha \cdot \beta$.
      \item $\lambda_1 (f)$ is the spectral radius and an eigenvalue of $f^*$.
  \end{enumerate}
\end{prop}
If $\lambda_1 (f) > 1$, then $\lambda_1$ is simple and there is a spectral gap property.
\begin{thm}[Theorem 3.5 of \cite{boucksomDegreeGrowthMeromorphic2008}]\label{ThmDynamicAutomorphismPicardManin}
  Let $f$ be a loxodromic automorphism of $\cM_D$, there exist nef elements $\theta_f^+, \theta_f^- \in \overline \cC
  (\cM_D)$ unique up to multiplication by a positive constant such that the following hold.
  \begin{enumerate}
  \item $\theta_f^+$ and $\theta_f^-$ are effective.
  \item $(\theta_f^+)^2 = (\theta_f^-)^2 = 0$, $ \theta_f^+ \cdot \theta_f^- = 1$.
  \item $f^* \theta_f^+ = \lambda_1 \theta_f^+$, $(f^{-1})^* \theta_f^- = \lambda_1 \theta_f^-$.
  \item For any $\alpha \in \overline \cZ (\cM_D)$,
    \begin{equation}
      \frac{1}{\lambda_1^N} (f^{\pm N})^* \alpha = \left(\theta_f^\mp \cdot \alpha\right) \theta_f^\pm + O_\alpha
      \left(\frac{1}{\lambda_1^N}\right).
      \label{<+label+>}
    \end{equation}
    \label{<+label+>}
  \end{enumerate}
\end{thm}
\begin{proof}
  The only assertion not following from \cite{boucksomDegreeGrowthMeromorphic2008} is that $\theta_f^+$ and
  $\theta_f^-$ are effective and belong to $\overline \cC$. We show the result for $\theta_f^+$. Following the proof of
  \cite{boucksomDegreeGrowthMeromorphic2008} Theorem 3.2, we have that $\theta_f^+$ is obtained as a limit of a
  sequence of Cartier divisors $(X_n, \theta_n)$ on compactifications of $\cM_D$ such that $f_{X_n}^* \theta_n = \rho_n
  \theta_n$ where $\rho_n$ is the spectral radius of $f_{X_n}^*$. The existence of $\theta_f^+$ follows from the fact that
  $\rho_n \rightarrow \lambda_1$. We show that $\theta_n$ can be chosen nef, effective and in $\DivInf(X_n)_\R$. Fix a
  compactification $X$ of $\cM_D$, $f_X^*$ preserves the nef cone $\Nef(X)$ of $\NS (X)$. It also preserves the subcone of
  $\Nef (X)$ consisting of effective divisors supported at infinity. This subcone is nonempty because for example in the
  compactification $\overline \cM_D \subset \P^3$, the Cartier divisor
  \begin{equation}
    H = \left\{ X = T = 0 \right\} + \left\{ Y = T = 0 \right\} + \left\{ Z = T = 0 \right\}
    \label{<+label+>}
  \end{equation}
  is very ample as it is equal to $\overline \cM_D \cap \left\{ T = 0 \right\}$. By Perron-Frobenius theorem, there exists
  $\theta_X$ in this subcone such that $f_X^* \theta_X = \rho_X \theta_X$ with $\rho_X$ the spectral radius of $f_X^*$.
\end{proof}

\subsection{Compatibility with adelic divisors}\label{SubSecCompatibilityAdelicDivisor}
We have a forgetful group homomorphism
\begin{equation}
  w: \hatDivInf(\cM_D / \OO_\K) \rightarrow \cW(\cM_D)
  \label{<+label+>}
\end{equation}
defined as follows. Let $\cU$ be a quasiprojective model of $\cM_D$ over $\OO_\K$ and let $\overline \sD$ be a model adelic
divisor on $\cU$. Then, $w (\overline \sD) = \sD_{\K}$ is the horizontal part of $\sD$, this is an element of $\cC
(\cM_D)$ because $\overline \sD$ is supported at infinity.

\begin{prop}
  The group homomorphism $w$ extends to a continuous group homomorphism
  \begin{equation}
    w : \hatDivInf(\cM_D / \OO_\K) \rightarrow \cW (\cM_D)
    \label{<+label+>}
  \end{equation}
  such that if $\overline D$ is integrable then $w (\overline D) \in \overline \cC (\cM_D)$ and if $f \in \Aut(\cM_D)$,
  then
  \begin{equation}
    w(f^* \overline D) = f^* w (\overline D).
    \label{eq:functioriality}
  \end{equation}
\end{prop}
\begin{proof}
  Let $\overline D \in \hatDivInf (\cM_D / \OO_\K)$ be given by a Cauchy sequence of model adelic divisors $(\overline
  \sD_i)$. Let $X$ be a compactification of $\cM_D$. There exists a sequence $\epsilon_i$ converging to zero such that
  \begin{equation}
  - \epsilon_i \overline \sD_0 \leq \overline \sD_j - \overline \sD_i \leq \epsilon_i \overline \sD_0
    \label{<+label+>}
  \end{equation}
  Applying $w$, we get (write $D_j = w(\overline \sD_j)$)
  \begin{equation}
    -\epsilon_i D_0 \leq D_j - D_i \leq \epsilon_i D_0
    \label{<+label+>}
  \end{equation}
  Thus, for every $E \in \sD_\infty, \ord_E (D_i)$ is a Cauchy sequence and converges to a number $\ord_E (D)$. By Lemma
  \ref{LemmeWeakTopology} this defines a Weil divisor $w (D) \in \cW (\cM_D)$. It is clear that $w$ is continuous, again
  using Lemma \ref{LemmeWeakTopology}.

  If $\overline D$ is integrable, then it is the difference of two strongly nef adelic divisors and nef classes in $\cW
  (\cM_D)$ belong to $\cC (\cM_D)$.

  If $f \in \Aut (\cM_D)$, it suffices to show \eqref{eq:functioriality} for model adelic divisors. Let $\overline \sD$
  be a model adelic divisor and let $\sX$ be a projective model of $\cM_D$ over $\OO_\K$ where $\sD$ is defined. There
  exists a projective model above $\sY$ of $\cM_d$ over $\OO_\K$ such that the lift of $f$ extends to a regular map
  $F : \sY \rightarrow \sX$ and $f^* \overline \sD = \overline{F^* \sD}$. Looking at the generic fiber we get
  \eqref{eq:functioriality}.
\end{proof}
We will drop the notation $w(\overline D)$ and just write $D = w(\overline D)$.

\section{Representation theory}\label{SecRepresentationTheory}
\subsection{Character varieties and the Markov surfaces}\label{SubsecCharactervarieties}
Let $\T_1$ be the once punctured torus. The fundamental group $\pi_1 (\T_1)$ is a free group generated by
two elements $a$ and $b$. The commutator $[a,b] := ab  a^{-1} b^{-1}$ is represented by a simple loop
around the puncture that follows the orientation of the surface. One can study the representation of
$\pi_1 (\T_1)$ into the affine variety
$\SL_2 (\C)$. It is clear that
\begin{equation}
  \hom (\pi_1 (\T_1), \SL_2 (\C)) \simeq \SL_2 (\C) \times \SL_2 (\C)
  \label{<+label+>}
\end{equation}
as $\pi_1 (\T_1)$ is a free group on two generators, therefore it is an affine variety. Define the character variety,
\begin{equation}
  \cX := \hom (\pi_1 (\T_1), \SL_2 (\C)) / / \SL_2 (\C)
  \label{<+label+>}
\end{equation}
where the action of $\SL_2 (\C)$ is diagonal and given by conjugation and $/ /$ is the Geometric
Invariant Theory (GIT) quotient. This is also an affine
variety and we have the following result of Fricke and Klein.

\begin{thm}[Fricke, Klein, \cite{goldmanTraceCoordinatesFricke2009}]
  The algebraic variety $\cX$ is isomorphic to $\C^3$. The isomorphism is given by
  \begin{equation}
    [\rho] \in \cX \mapsto (\Tr (\rho (a)), \Tr (\rho (b)), \Tr (\rho (ab))).
    \label{<+label+>}
  \end{equation}
\end{thm}
We will denote by $(x,y,z) = (\Tr (\rho (a)), \Tr (\rho (b)), \Tr (\rho (ab)))$ these are the
\emph{Frick-Klein} coordinates.
Let $\kappa : \cX \rightarrow \C$ be the
regular function
\begin{equation}
  \kappa (\rho) = \Tr (\rho([a,b]))
  \label{<+label+>}
\end{equation}
where $[a,b] = ab a^{-1} b^{-1}$.
One can show that
\begin{equation}
  \kappa = x^2 + y^2 + z^2 - xyz - 2
  \label{<+label+>}
\end{equation}

Therefore, if $\cX_t = \kappa^{-1} (t)$ is the relative character variety, we have
\begin{equation}
  \cX_t = \cM_{t+2}
  \label{<+label+>}
\end{equation}
where $\cM_{t+2}$ is the Markov surface of parameter $t+2$. In particular, the parameter $D = 4$ corresponds to
$t = 2$ and every points in $\cM_4$ corresponds to a reducible representation.
\subsection{Automorphism group of the Markov surfaces}\label{SubSecIntroAutGroupOfMarkovSurface}

The generalized mapping class group $\Mod^* (\T_1)$ is the group of homotopy class of homeomorphism
of $T_1$ (not necessarily orientation preserving). It contains $\Mod(\T_1)$ as an index 2 subgroup and
it acts on $\pi_1 (\T_1)$, we have the following
isomorphism:
\begin{equation}
  \Mod^*(\T_1) \simeq \Out (\pi_1 (\T_1))
  \label{<+label+>}
\end{equation}

Furthermore,
\begin{equation}
  \Out (\pi_1 (\T_1)) \simeq \GL_2 (\Z).
  \label{<+label+>}
\end{equation}
For any element $\Phi \in \Out (\pi_1 \left( \T_1 \right)), \Phi ([a,b])$ is conjugated to $
[a,b]^\pm$. This implies, that the action of $\Mod^* (\T_1)$ on $\cX$ preserves every $\cX_t$. Now,
the matrix $-\id$ acts trivially, because in $\SL_2 (\C)$ we have that $\Tr A
= \Tr A^{-1}$, so for all $D \in \C$ we get a group homomorphism
\begin{equation}
  \PGL_2 (\Z) \rightarrow \Aut (\cM_D)
  \label{EqGroupHomoGL2AutMarkov}
\end{equation}

\begin{thm}[\cite{cantatHolomorphicDynamicsPainleve2007} Theorem A and B,
  \cite{el-hutiCubicSurfacesMarkov1974}]\label{ThmAutMarkov}
  Let $\Gamma^* \subset \PGL_2 (\Z)$ be the subgroup of element congruent to the identity modulo 2, then for
  any $D \in \C$,
  \begin{equation}
    \Gamma^* \rightarrow \Aut (\cM_D)
    \label{<+label+>}
  \end{equation}
  is injective and its image is of index at most 8. Furthermore if $\Phi \in \Gamma^*$ and $\rho$ is its spectral
  radius, then $\rho = \lambda_1 (f_\Phi)$ where $f_\Phi$ is the automorphism of $\cM_D$ induced by $\Phi$.
  Furthermore, the topological entropy of $f_\Phi$ is equal to $\log \lambda_1 (f_\Phi)$.
\end{thm}
We can describe the group homomorphism. Let $\sigma_x \in \Aut (\cM_D)$ be the automorphism
\begin{equation}
  \sigma_x (x,y,z) = (yz -x, y, z),
  \label{<+label+>}
\end{equation}
If we fix the coordinates $y,z$, then the equation defining $\cM_D$ becomes a polynomial equation of
degree 2 with respect to $x$, $\sigma_x$ permutes the 2 roots of this equation. We can define
$\sigma_y, \sigma_z$ in the same way. Then, $\sigma_x, \sigma_y, \sigma_z$ generate a free group
isomorphic to $(\Z /2 \Z) * (\Z /2 \Z) * (\Z /2 \Z)$ which is of finite index in $\Aut (\cM_D)$
(see \cite{el-hutiCubicSurfacesMarkov1974}). The subgroup $\Gamma^*$ is the free group on the three
generators
\begin{equation}
  \begin{pmatrix}-1 & -2 \\ 0 & 1 \end{pmatrix}, \quad \begin{pmatrix}
    1 & 0 \\ -2 & -1
  \end{pmatrix}, \quad \begin{pmatrix}
    1 & 0 \\ 0 & -1
  \end{pmatrix}
  \label{EqMatrixOfGenerators}
\end{equation}
which correspond respectively to $\sigma_x, \sigma_y, \sigma_z$. For a more detailed description of the action of
$\GL_2 (\Z)$ on the character variety, see the appendix of \cite{goldmanModularGroupAction2003}.
\subsection{Fuchsian and Quasi-Fuchsian representation} The general reference for this part is
\cite{mcmullenRenormalization3ManifoldsWhich1996}.
Let $\HH^n$ be the hyperbolic space of dimension $n$. Recall that $\Isom(\HH^2) = \PSL_2 (\R)$ and $\Isom (\HH^3) =
\PSL_2(\C)$. The boundary of $\HH^2$ is naturally homeomorphic to $\CS^1 = \P^1 (\R)$ and the boundary of $\HH^3$ is naturally
homeomorphic to $\P^1 (\C) := \hat \C$.
A \emph{Fuchsian} group is a discrete subgroup $\Gamma$ of $\PSL_2 (\R)$. A \emph{Quasi-Fuchsian}
group is a discrete subgroup $\Gamma$ of $\PSL_2 (\C)$ such that its limit set \footnote{The limit set of $\Gamma$ is
the set of accumulation points in $\hat \C$ of $\Gamma \cdot x$ where $x \in \HH^3$ is any point.} in $\hat \C$ is a Jordan curve. Let $S$
be an oriented compact (real) surface of negative Euler characteristic. We say that
a representation $\rho: \pi_1 (S) \rightarrow \SL_2 (\C)$ is Fuchsian (resp, Quasi-Fuchsian) if
$\overline \rho (S) \subset \PSL_2 (\C)$ is Fuchsian (resp. Quasi-Fuchsian). We will denote by $\QF(S)$ the set of
Quasi-Fuchsian representation of $\pi_1(S)$ in $\PSL_2 (\C)$.

The \emph{Character variety} of $S$ is the algebraic variety $\cX(S) := \hom (\pi_1 (S) , \SL_2 (\C)) // \SL_2
(\C)$ where $/ /$ is the GIT quotient by the action of $\SL_2(\C)$ by conjugation.
The \emph{mapping class group} $\Mod (S)$ of $S$
is the group of orientation preserving homeomorphism of $S$ modulo the ones homotopic to the identity. It acts on the
Character variety.

Let $\Teich (S)$ be the Teichmuller space of $S$, that is the set of complete finite
hyperbolic metrics over $S$. The mapping class group $\Mod(S)$ acts on it. Every point of $\Teich(S)$ induces a Fuchsian
representation of $S$. We can actually parametrize the set of Quasi-Fuchsian representations $\QF(S)$ of $S$ using $\Teich(S)$ by
the simultaneous uniformization theorem of Bers.

\begin{thm}[\cite{bersSimultaneousUniformization1960}]
  There is a biholomorphic map
  \begin{equation}
    \Bers: \Teich(S) \times {\Teich(\overline S)} \rightarrow \QF(S)
    \label{<+label+>}
  \end{equation}
where $\overline S$ is the surface $S$ with its reversed orientation.
\end{thm}

Using this theorem, one can apply an iterative process to find a fixed point in the character
variety of $S$.

\begin{thm}[\cite{mcmullenRenormalization3ManifoldsWhich1996}] \label{ThmHyperbolisation}
  Let $S$ be an oriented compact surface of negative Euler characteristic.
  Let $(X,Y) \in \Teich (S) \times {\Teich (\overline S)}$, let $\Phi \in \Mod(S)$ be pseudo-Anosov,
  then the sequence
  \begin{equation}
    \Bers(\Phi^n(X), \Phi^{-n}(Y))
    \label{<+label+>}
  \end{equation}
  has an accumulation point $\rho_\infty: \pi_1 (S) \rightarrow \PSL_2 (\C)$. Furthermore,
  \begin{enumerate}
    \item $\rho_\infty$ is discrete and faithful.
    \item The limit set of $\rho_\infty (\pi_1 (S))$ is the whole boundary $\hat \C$ of $\HH^3$.
    \item $\rho_\infty$ is a fixed point of $\Phi$ and $\Phi$ is conjugated to an isometry $\alpha$
      of $\tilde M_\Phi = \HH^3 / \rho_\infty (\pi_1 (S))$.
    \item The group of isometries of $M_\infty$ is discrete and $\alpha$ is of infinite order.
    \item The mapping torus
      \begin{equation}
        M_\Phi := S \times [0,1] / (x,0) \sim (\phi(x), 1)
      \end{equation}
      is isomorphic as an hyperbolic manifold
      to $\tilde M_\Phi / < \alpha >$.
    \item The subgroup generated by $\alpha$ of the group of isometries of $\tilde M_\Phi$ is of
      finite index.
    \item The fixed point $\rho_\infty$ of $\psi$ is hyperbolic, meaning that no eigenvalues of the differential of
      $\psi$
  \end{enumerate}
\end{thm}

\subsection{The surface $\cM_0$ and a Theorem of Minsky}
We are interested in this section with the Markov surface $\cM_0$ that is when $\kappa = -2$,
therefore $\rho (aba^{-1}b^{-1})$ is a parabolic Möbius transformation. We recall the results from
\cite[\S 4.1]{cantatBersHenonPainleve2009}. The real points $\cM_0 (\R)$ consist of an
isolated point $(0, 0, 0)$ and four diffeomorphic connected components that are given by the signs
of $x$ and $y$. We will denote by $\cM_0 (\R)^+$ the connected component such that $x, y > 0 $.
of area $2 \pi$. This is equivalent to asking that there is a cusp at
the puncture. It is known that $\Teich (\T_1)$ ($\T_1$ being the punctured torus) is isomorphic to the
upper half-plane $\HH^+$ and we make this identification from now on. The action of $\Mod (\T_1)$ on
$\Teich(\T_1)$ is conjugated to the usual action of $\PSL_2 (\Z)$ by isometries on $\HH^+$. The point $(0,0,0)$ is the only
singular point of $\cM_0$ and it corresponds to the conjugacy class of the representation
\begin{equation}
  \rho (a) = \begin{pmatrix}
    0 & i \\
    i & 0
  \end{pmatrix}, \quad \rho (b) = \begin{pmatrix}
    0 & -1 \\
    1 & 0
  \end{pmatrix}.
  \label{<+label+>}
\end{equation}
Its image is the quaternionic group of order eight. The automorphism group $\Aut (\cM_0)$ fixes $(0,0,0)$, preserves
$\cM_0 (\R)^+$ and permutes the three remaining connected components of $\cM_0 (\R) \setminus (0,0,0)$.

Any point in $\Teich (\T_1)$ gives rise to a representation $\overline \rho: \pi_1 (\T_1) \rightarrow
\PSL_2 (\R)$ which can be lifted to four distinct representations $\rho: \pi_1 (\T_1) \rightarrow
\SL_2 (\R)$. The cusp condition gives the condition $\Tr (\rho (a,b)) = -2$ (because $\Tr = 2$
corresponds to reducible representations). Therefore, we get an
embedding of $\Teich (\T_1)$ into the 4 different connected component of $\cM_0 (\R) \setminus
(0,0,0)$. We will restrict our attention to the embedding $\Teich (\T_1) \hookrightarrow \cM_0
(\R)^+$.

The set $\cM_0 (\R)^+$ is made of (conjugacy class of) Fuchsian representations. Let $\DF_0 \subset
\cM_0 (\C)$ be the subset of discrete and faithful representation of $\pi_1 (\T_1)$. Then $\DF_0$ has
four different connected components, one of them contains $\cM_0(\R)^+$. We denote it by $\DF_0^+$ and
we denote by $\QF_0^+$ the set of Quasi-Fuchsian representation inside $\DF_0^+$. In fact, $\QF_0^+$
is the interior of $\DF_0^+$ (see \cite{minskyEndInvariantsClassification2002}).
We can identify $\Teich(\T_1)$ with the upper half plane $\HH^+$ and ${\Teich (\overline \T_1)}$ with
the lower half plane $\HH^-$. The group $\PSL_2 (\Z)$ acts on $\P^1 (\C)$ via Möbius transformation.
It preserves $\HH^+, \HH^-$ and $\P^1 (\R)$. In particular, the mapping class group $\Mod (\T_1) =
\SL_2 (\Z)$ acts on $\P^1 (\C)$ and we can conjugate this action to the action on $\cM_0 (\R)^+$ via
the Bers mapping. Namely, let $\Phi \in \Mod (\T_1)$ and let $f_\Phi \in \Aut (\cM_0)$ induced by
the map from Equation \eqref{EqGroupHomoGL2AutMarkov}. We have for every $(s,t) \in \HH^+ \times
\HH^-$,
\begin{equation}
  \Bers (\Phi (s,t)) = f_\Phi (\Bers (s,t))
  \label{}
\end{equation}

Theorem \ref{ThmHyperbolisation} is not applicable directly as $\T_1$ is not compact. However,
Minsky showed that the Bers mapping can be extended to almost all the boundary of
$\Teich(\T_1) \times \Teich (\overline \T_1)$. The boundary of $\HH^+$ is $\P^1 (\R)$. We denote by
$\Delta$ the diagonal in $\partial \Teich(\T_1) \times \partial \Teich(\overline \T_1)$.

\begin{thm}[\cite{minskyClassificationPuncturedTorusGroups1999}]
  The Bers mapping extends to a continuous bijection
  \begin{equation}
    Bers: \overline{\Teich (\T_1) \times \Teich (\overline \T_1)} \setminus \Delta \rightarrow \DF^+
    \label{<+label+>}
  \end{equation}
\end{thm}

In particular, let $\Phi \in \SL_2 (\Z) = \Mod (\T_1)$ be a loxodromic element and let $f_\Phi$ be
its associated automorphism over $\cM_0$. The isometry $\Phi$ has a repulsive fixed point
$\alpha(\Phi)$ on $\P^1 (\R)$ and an attractive one $\omega(\Phi)$. By Minsky's theorem, this gives two
unique fixed point $p(\Phi) = \Bers ( (\alpha(\Phi), \omega(\Phi)))$ and $q (\Phi) = \Bers ( (\omega(\Phi),
\alpha(\Phi)))$ of $f_\Phi$ in $\DF^+ \setminus \QF^+$. This two fixed points are also hyperbolic because any point in
$D_F^+$ converges exponentially fast to $q(\phi)$ in forward time and to $p(\phi)$ in backward time so we can apply
Corollary 3.19 of \cite{mcmullenRenormalization3ManifoldsWhich1996}.

\subsection{The surface $\cM_D$ for $D = 2 - 2 \cos (\frac{\pi}{q})$}\label{subsec:orbifold}
The main reference for this part is \cite[\S3.7]{mcmullenRenormalization3ManifoldsWhich1996}.
Let $q \geq 2$ be an integer. Consider the orbifold real compact surface $S_q$ which is a genus 1 torus with a singular
point of index $q$. The modular class group of $S$ is still $\SL_2 (\Z)$. Let $\psi \in \SL_2 (\Z)$ be pseudo-Anosov,
i.e such that $\left|\Tr \psi \right| > 2$. We can recover Theorem \ref{ThmHyperbolisation} for $S_q$ as follows. There
exists a real closed compact surface $\tilde S_q$ with a finite characteristic cover $\tilde S_q \rightarrow S_q$. In
particular, $\Mod(S_q) = \Mod (\tilde S_q)$ so we can apply Theorem \ref{ThmHyperbolisation} for $\tilde \psi$, the lift
of $\psi$ to $\tilde S_q$. We get a hyperbolic 3-fold $M_{\tilde \psi}$ which is topologically the mapping torus of
$\tilde \psi$. We have a finite cover $M_{\tilde \psi} \rightarrow M_{\psi}$ and by Mostow's rigidity theorem this
induces a hyperbolic metric on $M_\psi$. 

Now the orbifold fundamental group of $S_q$ is 
\begin{equation}
  \pi_1 (S_q) = \left< a,b | [a,b]^q = 1 \right>.
  \label{<+label+>}
\end{equation}
Every irreducible representation $\rho : \pi_1 (S) \rightarrow \PSL_2 (\C)$ lifts to a representation 
\begin{equation}
  \overline \rho : G \rightarrow \SL_2 (\C)
  \label{<+label+>}
\end{equation}
where $G = \left< a,b | [a,b]^{2q} = 1 \right>$ such that 
\begin{equation}
  \overline \rho ([a,b]^q) = - \id.
  \label{<+label+>}
\end{equation}
In particular, $\rho ([a,b])$ is an elliptic element of order $2q$ thus 
\begin{equation}
  \Tr (\rho ([a,b])) = -2 \cos \left( \frac{2 \pi}{2q} \right) = - 2 \cos \left(\frac{\pi}{q}\right).
  \label{<+label+>}
\end{equation}
We get that $\overline \rho \in \cM_{2 -2 \cos \left(\frac{\pi}{q}\right)}$. In particular the representation induced by
$M_\psi$ yields a hyperbolic fixed point of $\psi$ over $\cM_D$.

\begin{rmq}\label{rmq:typo-mcmullen}
  In \cite[\S3.7]{mcmullenRenormalization3ManifoldsWhich1996}, it is stated that every irreducible representation
  $\overline \rho$ is a point of $M_D$ for $D = 2 - 2 \cos(2 \pi /q)$ but we believe this is a typo as for $q=2$ we would
  get $D=4$ and there are no irreducible representations in $\cM_4$. The fixed point $M_\psi$ constructed in
  \cite[\S3.7]{mcmullenRenormalization3ManifoldsWhich1996} in the example for $q=2$ is actually a point in $\cM_0$ and
  not $\cM_4$. See also \cite[Theorem 1.1]{cantatBersHenonPainleve2009}.
\end{rmq}

\section{Dynamics of loxodromic automorphisms of the Markov surface}
\subsection{Cyclic compactifications and circle at infinity}
Let $\overline \cM_D \subset \P^3$ be the closure of $\cM_D$ in $\P^3$. We have that $\overline \cM_D \setminus \cM_D$ is a triangle
of lines. We use the following result.

\begin{prop}[\cite{el-hutiCubicSurfacesMarkov1974,
  cantatCommensuratingActionsBirational2019}]\label{PropCycleCaseIndeterminacyPoints}
  Let $X$ be projective surface and $U$ an open subset of $X$ such that $X \setminus U$ is a cycle
  of rational curves. Assume that $X \setminus U$ is not an irreducible curve with one nodal
  singularity. Let $g$ be an automorphism of $U$, then the indeterminacy points of $g$ can only be
  intersection points of two components of the cycle.
\end{prop}

This shows that to understand the dynamics of a loxodromic automorphism at infinity it suffices to blow up the
intersection points of the prime divisors at infinity. Therefore we can remain with compactifications $X$ of $\cM_D$ such that $X
\setminus \cM_D$ is a cycle of rational curves. We call them cyclic compactifications.

Start with $\overline \cM_D \subset
\P^3$. If we blow up the three intersection points of the triangle at infinity we get a new compactification of $\cM_D$ with
a hexagon of rational curves at infinity. If we repeat the process we get a sequence of compactifications with an increasing
polygon of rational curves at infinity, let $X_n$ be this sequence of compactifications. Let $\cG_n$ be the dual graph of $X
\setminus \cM_D$, i.e the vertices of $\cG_n$ are the irreducible components $E_i$ of $X_n \setminus \cM_D$ and we have an
edge between $E_i$ and $E_j$ if and only if $E_i \cap E_j \neq \emptyset$. We have a natural embedding of graphs $\cG_n
\hookrightarrow \cG_{n+1}$. We define the direct limit $\cG = \varinjlim_{n} \cG_n$ and $\Aut(\cM_D)$ acts naturally
on $\cG$. There is a
parametrization of the vertices of $\cG$ called the \emph{Farey parametrization} (see
\cite{cantatCommensuratingActionsBirational2019} \S 8.2), and $\cG$ is isomorphic to the set of rational points of the
circle $\CS^1$ with this parametrization. If $X$ is a cyclic compactification, the irreducible components $E_i$ (enumerated
cyclically) of the boundary corresponds to rational points $v_i \in \CS^1$. The rational points of $]v_i, v_{i+1}[$ are
  obtained by blowing up above the point $E_i \cap E_{i+1}$.

  Following $\S 2.6$ of \cite{cantatBersHenonPainleve2009}, we identify $\CS^1 \simeq \R \cup \left\{ \infty \right\}$ with the
  boundary of the upper half plane of
$\HH^+$. If $v_x, v_y, v_z$ are the
three vertices of $\cG$ representing the three curves $\left\{ X=T=0 \right\}$, $\left\{ Y=T=0 \right\}$ and $\left\{ Z = T= 0
\right\}$ in $\overline \cM_D$, then this identification sends $v_x, v_y, v_z$ to $0, -1, \infty \in \partial \HH^+$ which
we write $j_x, j_y, j_z$ respectively. Recall the notations of \S \ref{SubSecIntroAutGroupOfMarkovSurface}, The
generator $\sigma_x$ (resp. $\sigma_y, \sigma_z$) of $\Gamma^*$ acts on
$\partial \HH^+$ as the reflection with respect to the geodesic $(j_y, j_z)$ (resp. $(j_x, j_z), (j_x, j_y)$) under this
identification. The isometries of $\HH^+$ induced by $\sigma_x, \sigma_y, \sigma_z$ with this identification are exactly
given by \eqref{EqMatrixOfGenerators}. Hence, the action of $\Gamma^*$ on $\cG$ is given by the action of $\Gamma^*$ on
$\HH^+$ via isometries. Thus every loxodromic automorphism $f \in \Gamma^*$ admits two
irrational fixed points $\alpha(f), \omega(f) \in \partial \HH^+$, $\alpha(f)$ is repulsive and $\omega(f)$ is
attracting.

The circle $\CS^1$ has an interpretation as a special subset of the set of valuations of the ring of regular functions
of $\cM_D$. The two fixed point $\alpha(f), \omega (f)$ correspond to a repulsive and attracting fixed point in this
space of valuation. See \S 14.5 of \cite{abboudDynamicsEndomorphismsAffine2023} for more details.

\subsection{Dynamics of loxodromic automorphism at infinity} From the previous discussion, we get the following proposition.
\begin{prop}\label{PropDynamicsLoxodromicAutomorphism}
  Let $D \in \C, \K = \Q(D)$ and $f \in \Aut (\cM_D)$. For any cyclic compactification $Y$ of $\cM_D$, there
  exists a cyclic compactification $X$ above $Y$ such that
  there exists two closed points $p_+, p_- \in (X \setminus \cM_D) (\K)$ satisfying the following properties
  \begin{enumerate}
      \item $p_+ \neq p_-$
    \item some positive iterate of $f^{\pm 1}$ contracts $X \setminus \cM_D$ to $p_\pm$.
    \item $f^\pm$ is defined at $p_{\pm}, f^{\pm 1} (p_\pm) = p_\pm$ and $p_\mp$ is the unique
      indeterminacy point of $f^{\pm N}$ for $N$ large enough.
  \item There exists local algebraic coordinates $(u,v)$ at $p_\pm$ such that $uv = 0$ is a local equation of the
    boundary and $f^{\pm 1}$ is locally of the form
    \begin{equation}
      f^{\pm 1} (u,v) = \left( u^a v^b \phi, u^c v^d \psi \right)
      \label{<+label+>}
    \end{equation}
    with $ad - bc = \pm 1$ and $\phi, \psi$ invertible.
    In particular, for any place $v$ of $\K$ there exists
    an open neighbourhood $U^\pm_{v}$ of $p_\pm$ in $X^{\an}_{v}$ such that $f^\pm (U^\pm_{\K_v}) \Subset
    U^\pm_{v}$.
  \item $f$ is algebraically stable over $X$ and
    \begin{equation}
      f_X^* \theta_{f,X}^+ = \lambda_1 \theta_{f,X}^+, \quad (f_X^{-1})^* \theta_{f,X}^- = \lambda_1.
      \theta_{f,X}^-
    \end{equation}
\end{enumerate}
\end{prop}
\begin{proof}
  From the previous section, write $E_1, \cdots, E_r$ for the irreducible components of $Y$ enumerated in cyclic order.
  We write $p_+$ for the unique point $ p_+ = E_i \cap E_{i+1}$ such that
$\omega(f) \in ]v_i, v_{i+1}[$ and we define $p_-$ similarly with respect to $\alpha(f)$. It is clear that for $N$ large
  enough every $E_i$ is contracted to $p_\pm$ by $f^{\pm N}$ because of the attractingness of $\omega(f)$ and
  $\alpha(f) = \omega(f^{-1})$. Thus, to get a compactification above $Y$ that satisfy Properties (1)-(5) we just need that
  $p_+ \neq p_-$ in that compactification. Since $f$ is loxodromic, we have $\alpha (f) \neq \omega(f)$, thus after enough
  blow-ups this will be the case. The fact that $f$ is algebraically stable follows from the fact that the action of
  $f$ over $\cG \simeq \cS^1 \simeq \P^1 (\R)$ is induced by a loxodromic Mobius transformation.
\end{proof}

\begin{cor}[Corollary 3.4 from \cite{cantatBersHenonPainleve2009}]
  Any loxodromic automorphism of $\cM_D$ does not admit any invariant algebraic curves.
\end{cor}
\begin{proof}
  Let $f \in \Aut(\cM_D)$ be loxodromic. Let $X$ be a compactification of $\cM_D$ given by Proposition
  \ref{PropDynamicsLoxodromicAutomorphism}.
  If $C \subset \cM_D$ was an invariant algebraic curve, then its closure $\overline C$ in $X$ should intersect
  $X \setminus \cM_D$. Since the boundary is contracted by $f$, we must have $p_+ \in \overline C$ and $f : \overline C
  \rightarrow C$ is an automorphism with a superattractive fixed point. This is a contradiction.
\end{proof}

\begin{prop}\label{PropOperateurPullBack}
  Let $X$ be a compactification of $\cM_D$ given by Proposition \ref{PropDynamicsLoxodromicAutomorphism}, replace $f$ by one of
  its iterates such that $f^{\pm 1}$ contracts $X \setminus \cM_D$ to $p_\pm$. Then,
  \begin{enumerate}
    \item There exists $D^- \in \DivInf(X)_\R$ such that $f_X^* D^- = \frac{1}{\lambda_1} D^-$.
    \item If $E_+, F_+$ are the two prime divisors at infinity such that $p_+ = E_+ \cap F_+$, then for all $R \in
      \DivInf(X)_\R$ such that $E_+, F_+ \not \in \Supp R$, then $f_X^* R = 0$ and $\theta_f^- \cdot R = 0$.
    \item $D^- \cdot \theta_f^- = 0$.
    \item $\{\theta_{f,X}^+, D^-\} \cup \left\{ E : E \not \in \left\{ E_+, F_+ \right\} \right\}$ is a basis of
      $\DivInf(X)$.
  \end{enumerate}
\end{prop}
\begin{proof}
  If $E$ is a prime divisor at infinity distinct from $E_+$ and $F_+$, then $f_X^* E = 0$ because every prime divisor at
  infinity is contracted to $p_+$ by $f$ and $p_+ \not \in E$.  Now if $R$ satisfies $f_X^* R = 0$, then
  \begin{equation}
    0 = f_X^* R \cdot \theta_{f,X}^- = R \cdot (f_X^{-1})^* \theta_{f,X}^- = \lambda_1 R \cdot \theta_{f,X}^-.
    \label{<+label+>}
  \end{equation}
  Thus $R \cdot \theta_{f,X}^- = 0$. This shows (2).

  Since $f_X^* \theta_{f,X}^+ = \lambda_1
  \theta_{f,X}^+$, we have that
  \begin{equation}
    \theta_{f,X}^+ = \alpha E_+ + \beta F_+ + \cdots
    \label{<+label+>}
  \end{equation}
  where $(\alpha, \beta)$ is an eigenvector of $A = \begin{pmatrix}
    a & c \\
    b & d
  \end{pmatrix}$ of eigenvalue $\lambda_1$. Now, up to replacing $f$ by $f^2$ we can assume that $ad - bc = 1$ and that
  the other eigenvalue of $A$ is $\frac{1}{\lambda_1}$. Let
  $(\gamma, \delta)$ be an associated eigenvector, then
  \begin{equation}
    f_X^* (\gamma E_+ + \delta F_+) = \frac{1}{\lambda_1} (\gamma E_+ + \delta F_+) + R
    \label{<+label+>}
  \end{equation}
  where $R$ is a divisor at infinity which support does not contain $E_+$ or $F_+$. Set $D^- = \gamma E_+ + \delta F_+ +
  \lambda_1 R$, then by (2), $D^-$ satisfies $f_X^* D^- = \frac{1}{\lambda_1} D^-$. This shows (1).

  Now,
  \begin{equation}
    \frac{1}{\lambda_1} D^- \cdot \theta_f^- = \frac{1}{\lambda_1} D^- \cdot \theta_{f,X}^- = (f_X^* D^-) \cdot \theta_{f,X}^- =
    D^- \cdot (f_{X})_* \theta_{f,X}^- = \lambda_1 D^- \cdot \theta_f^-.
    \label{<+label+>}
  \end{equation}
  Thus $D^- \cdot \theta_f^- = 0$. This shows (3).

  Finally, we just have to show that the family $\left\{ \theta_{f,X}^+, D^- \right\} \cup \left\{ E; E \not \in \left\{
  E_+, F_+ \right\} \right\}$ is free. Suppose that
  \begin{equation}
    \alpha \theta_{f,X}^+ + \beta D^- + R = 0
    \label{Eq2}
  \end{equation}
  with $\alpha, \beta \in \R$ and $E_+, F_+ \not \in \Supp R$. Intersecting with $\theta_f^-$ in \eqref{Eq2} and using (2)
  and (3), we get $\alpha = 0$. Then, applying $f_X^*$ to \eqref{Eq2} we get $\beta = 0$. Thus $R = 0$ and we have shown
  (4).
\end{proof}

\section{An invariant adelic divisor}\label{SecInvariantAdelicDivisor}
We now suppose that $D \in \C$ is algebraic and write $\K = \Q(D)$ which is a number field.
\begin{thm}\label{ThmInvariantAdelicDivisor}
  Let $f$ be a loxodromic automorphism of $\cM_D$, then there
  exists a unique, up to multiplication by a positive constant, adelic
  divisor $\overline \theta_f^+ \in \hatDivInf(\cM_D / \OO_\K)$ such that
  \begin{equation}
    f^* \overline \theta_f^+ = \lambda_1 \overline \theta_f^+.
    \label{<+label+>}
  \end{equation}
  Furthermore, $\overline \theta_f^+$ is a strongly nef and effective adelic divisor.
\end{thm}
\begin{rmq}
 Of course, the same result holds for $f^{-1}$ with the existence of a unique effective strongly nef invariant adelic
 divisor $\overline \theta_f^-$, the proof is symmetric. With the notations of \S \ref{SubSecCompatibilityAdelicDivisor},
 we must have $w(\overline \theta_f^\pm) = \theta_f^\pm$ (up to multiplication by a positive constant) because of Theorem
 \ref{ThmDynamicAutomorphismPicardManin}. So our notation of $\overline \theta_f^+$ is compatible with Theorem
 \ref{ThmDynamicAutomorphismPicardManin}.
\end{rmq}
The rest of this section is dedicated to the proof of this theorem. We write $\theta^\pm$ instead of $\theta^{\pm}_f$.

\subsection{Two lemmas}
Start with the following lemma.
\begin{lemme}\label{LemmeConvergenceModelVerticalDivisor}
  If $\overline \sD$ is a model vertical divisor, then
  \begin{equation}
    \frac{1}{\lambda_1^n} (f^n)^* \overline \sD \rightarrow 0
    \label{<+label+>}
  \end{equation}
in $\hatDivInf (\cM_D/ \OO_\K)$.
\end{lemme}
\begin{proof}
  Let $g$ be the Green function induced by $\overline \sD$ over $\cM_D^{\an}$. It is a continuous bounded function over
  $\cM_D^{\an}$. Therefore
  \begin{equation}
    \frac{1}{\lambda_1^n} g \circ (f^{\an})^n
    \label{<+label+>}
  \end{equation}
  converges uniformly to zero over $\cM_D^{\an}$. Thus it also converges to zero for the boundary topology.
\end{proof}

From now on, we fix a compactification $X$ of $\cM_D$ that satisfies Proposition \ref{PropDynamicsLoxodromicAutomorphism}. We
also replace $f$ by one of its iterates $s:= f^{N_0}$ such that $s^{\pm 1}$ contracts $X \setminus \cM_D$ to $p_\pm$. We will
show Theorem \ref{ThmInvariantAdelicDivisor} for $s$ and then deduce the result for $f$. Notice that we have
$\theta^\pm_s = \theta^\pm_f$ so we will keep the notation $\theta^\pm$.

\begin{lemme}\label{LemmeGenericallyInvariantGreenFunctions}
   Let $D \in \DivInf(X)_\R$ be a $\R$-divisor at infinity such that $s_X^* D = \mu D$ for some $\mu \in \R$. Let $(\sX,
   \sD)$ be a model of $(X, D)$ over $\OO_\K$. Let $V \subset \spec \OO_\K$ be an open subset such that
  \begin{enumerate}
    \item The indeterminacy locus $I^\pm$ of the rational map $s^{\pm 1} : \sX \dashrightarrow \sX$ satisfies
      \begin{equation}
        I^\pm_{|V} = (I \cap \overline{\left\{ p_\mp \right\}})_{|V}.
      \end{equation}
  \item $\sD$ is horizontal over $V$.
\end{enumerate}
Then, for every finite place $v$ above $V$, let $U^-_v$ be the open subset
\begin{equation}
  U^-_{v} := \left\{ x \in X^{\an}_v : r_{\sX_v} (x) = r_{\sX_v} (p_-) \right\}.
  \label{eq:<+label+>}
\end{equation}

Then, $s^{- 1}$ is defined over $U^-_{v}$, $U^-_v$ is $s^{- 1}$-invariant and if $W_v^- = (X_v^{\an}
\setminus U^-_v) \cap \cM_{D,v}^{\an}$, then
\begin{equation}
\label{eq:egalite-sur-W-mu-non-nul}
\left(g_{\left(\sX_v, \sD_v\right)} \circ s^{\an}\right)_{|W_v^-} = \mu \cdot {g_{\left(\sX_v, \sD_v\right)}}_{|W^-_v}.
\end{equation}
  \end{lemme}

  \begin{proof}
    Let $\pi : \sY \rightarrow \sX$ be the normalised blow up of the indeterminacy locus of $s : \sX \dashrightarrow
    \sX$ such that the lift $S : \sY \rightarrow \sX$ of
    $s$ is regular. We have that
    \begin{equation}
      g_{(\sX, \sD)} \circ s^{\an} = g_{(\sY, S^* \sD)}.
      \label{<+label+>}
    \end{equation}
    Recall that both sides of this equality are functions over $\cM_D^{\an}$ which is an open subset of both
    $X^{\an}$ and $Y^{\an}$.
    Now, over $V$ the birational morphism $\pi$ consists only of horizontal blow ups that is blow ups of $\overline p$
    where $p$ is a closed point of the generic fiber. Furthermore these points $p$ are all above $p_-$. By hypothesis,
    $\pi$ induces an isomorphism between $\sX_V \setminus \overline{\{ p_- \}}$ and $\pi^{-1} (\sX_V \setminus
    \overline{\{ p_- \}})$ because all the indeterminacy points of $s : X
    \dashrightarrow X$ are above $p_-$. Plus, looking at the horizontal part of our divisors we get that
    \begin{equation}
      E = (S^* \sD)_{\text{hor}} - \mu (\pi^* \sD)_{\text{hor}}.
      \label{eq:<+label+>}
    \end{equation}
    is an $\R$-Weil divisor over $Y$ that has support only over the exceptional divisors of $\pi$ (we also write $\pi: Y
    \rightarrow X$ for the sequence of blow-ups induced on the generic fiber) which are all above $p_-$. Thus we get
    that the $\R$-Weil divisor
    \begin{equation}
      \left( S^* \sD - \mu \pi^* \sD \right)_{| \pi^{-1} (\sX_V \setminus \overline{ \{p_-\}})}
      \label{eq:<+label+>}
    \end{equation}
    is trivial, hence we get \eqref{eq:egalite-sur-W-mu-non-nul} for every finite place $v$ over $V$ by Corollary
    \ref{cor:Green-function-zero-if-no-support}.
  \end{proof}

  \subsection{An iterative process}
  \begin{prop}\label{PropConvergenceModelDivisorSemiSimple}
    Let $D \in \DivInf(X)$ be such that $f_X^* D = \mu D$ for some $\mu \in \R$. Let $(\sX, \sD)$ be a model of
    $(X,D)$, then
    \begin{itemize}
      \item If $\mu = \lambda_1$, then $D = \theta_{f,X}^+$ up to renormalisation and $\frac{1}{\lambda_1^n} (s^n)^* \sD$
        converges towards an element $\overline \theta^+(X)$ such that $s^* \overline \theta^+ (X) = \lambda_1 \overline
        \theta^+(X)$ which a priori depends on $X$.
      \item Else $\left|\mu \right| < \left|\lambda_1\right|$ and $\frac{1}{\lambda_1^n} (s^n)^* \sD$ converges towards 0.
    \end{itemize}
  \end{prop}

  To prove the proposition, we study the sequence of adelic divisors
     \begin{equation}
       \frac{1}{\lambda_1^n} (s^n)^* \sD - \left(\frac{\mu}{\lambda_1}\right)^n \sD.
      \label{<+label+>}
    \end{equation}
    We show that if $|\mu| < \lambda_1$, then this sequence tends to $0$ and if $\mu = \lambda_1$, then this sequence
    converges towards an adelic divisor. Looking at Green functions, we need to show that the sequence
    \begin{equation}
      u_n := \frac{1}{\lambda_1^n} g_{(\sX, \sD)} \circ (s^n)^{\an} - \left(\frac{\mu}{\lambda_1}\right)^n g_{(\sX, \sD)}.
      \label{<+label+>}
    \end{equation}
    converges over $\cM_D^{\an}$ with respect to the boundary topology. We split the proof into two parts. First we show
    that away from the indeterminacy point $p_-$, the convergence is actually uniform. Then, we will study in more
    details what happens near $p_-$ and show the convergence for the boundary topology there.

    \subsection{Convergence away from $p_-$} \label{SubSecConvAwayFromPminus}
    Since $p_- \in X(\K)$ we write $p_-^v$ for its image in $X^{\an}_v$.
    There exists an open neighbourhood $U^- := \bigsqcup_v U^-_v$ of $\bigcup_v \{p_-^v\}$ in $X^{\an}$ such that
    $s^{-1}$ is defined over $U^-$ and $U^-$ is $s^{-1}$-invariant. Indeed,
    let $V \subset \spec \OO_\K$ be an open subset that satisfies the hypothesis of Lemma
    \ref{LemmeGenericallyInvariantGreenFunctions}. For every finite place $v$ above $V$ we set $U^-_v = \left\{ x :
      r_{\sX_v} (x) = r_{\sX_v} (p_-^v) \right\}$ as the open subset of
    $X_v^{\an}$ defined in Lemma \ref{LemmeGenericallyInvariantGreenFunctions}. For the
    finitely many remaining places $v$ (including also all the infinite ones), we know by Proposition
    \ref{PropDynamicsLoxodromicAutomorphism} that $p_-^v$ is an attracting fixed point of $(s^{-1}_v)^{\an}$. Let
    $U^-_v$ be an $(s^{-1}_v)^{\an}$ invariant open neighbourhood of $p_-^v$ such that $(s^{-1}_v)^{\an} (U_v^-) \Subset
    U_v^-$. Define
    \begin{equation}
    U^- := \bigsqcup_{v \in M(\K)} U^-_v.
  \end{equation}
    It is an $(s^{-1})^{\an}$
    invariant open subset. Let $W^-$ be its complement, it is $s^{\an}$-invariant.

    Set
    \begin{equation}
      h = u_1 = \frac{1}{\lambda_1} g_{(\sX, \sD)} \circ
      s^{\an} - \frac{\mu}{\lambda_1} g_{(\sX, \sD)}.
    \end{equation}

    \begin{lemme}\label{LemmeFonctionH}
      The function $h$ extends to a continuous bounded function over $W^-$ and
      \begin{equation}
        h_{W_{V[\fin]}^-} \equiv 0
        \label{<+label+>}
      \end{equation}
    \end{lemme}
    \begin{proof}
    First, by Lemma
    \ref{LemmeGenericallyInvariantGreenFunctions} we have $h \equiv 0 $ over $W_v^-$ for every finite place $v \in
    V[\fin]$.
    If $v \not \in V[\fin]$, then since $f_X^* D = \mu D$, we have that $h$ extends to a continuous function
    over $W^-_v$ because $p_-^v \not \in W^-_v$. Since $W^- \setminus W_{V[\fin]}^-$ is compact we have that $h$ is a bounded
    continuous function over $W^-$.
    \end{proof}

    \begin{prop}\label{PropConvergenceSurWmoins}
      If $|\mu| < \lambda_1$, then $u_n$ converges uniformly to $0$ over $W^-$.

      If $\mu = \lambda_1$, then $u_n$ converges uniformly towards a continuous function $h^+$ over $W^-$ such that
      \begin{enumerate}
        \item $h^+_{|W_{V[\fin]}^-} \equiv 0$.
        \item If $G^+ = h^+ + g_{(\sX, \sD)}$, then $G^+ \circ s^{\an} = \lambda_1 G^+$.
        \label{<+label+>}
      \end{enumerate}
    \end{prop}
    \begin{proof}We have
    \begin{align}
      u_n &= \frac{1}{\lambda_1^{n-1}} \left(\frac{1}{\lambda_1} g_{(\sX,\sD)} \circ s\right) \circ s^{n-1} - \left(
      \frac{\mu}{\lambda_1} \right)^n g_{(\sX,\sD)} \\
      &= \frac{1}{\lambda_1^{n-1}} \left( h + \frac{\mu}{\lambda_1} g_{(\sX,\sD)} \right) \circ s^{n-1} -
      \left(\frac{\mu}{\lambda_1}\right)^n g_{(\sX,\sD)} \\
      &= \frac{1}{\lambda_1^{n-1}} h \circ s^{n-1} + \frac{\mu}{\lambda_1} u_{n-1}.
      \label{<+label+>}
    \end{align}
    Therefore,
    \begin{equation}
      u_n = \sum_{\ell =0}^{n-1} \frac{\mu^\ell}{\lambda_1^{n-1}}  h \circ s^{n-1-\ell}.
      \label{Eq1}
    \end{equation}
    If $|\mu| < \lambda_1$, let $M = \max_{W^-} |h|$, then
    \begin{equation}
      \sup_{W^-} | u_n (x) | \leq \frac{M}{\lambda_1^{n-1}} \frac{|\mu|^n - 1}{|\mu| - 1} \xrightarrow[n \rightarrow +
      \infty]{} 0.
      \label{<+label+>}
    \end{equation}
    If $\mu = \pm 1$, then
    \begin{equation}
      \sup_{W^- } | u_n (x) | \leq \frac{M n}{\lambda_1^{n-1}} \xrightarrow[n \rightarrow + \infty]{} 0.
      \label{<+label+>}
    \end{equation}
    If $\mu = \lambda_1$, then write $\sD^+$ for $\sD$. Equation \eqref{Eq1} becomes
          \begin{equation}
            u_n = \sum_{k=0}^{n-1} \frac{1}{\lambda_1^k} h \circ s^k.
            \label{<+label+>}
          \end{equation}
          Since $h$ is bounded over $W^-$, $u_n$ converges uniformly over $W^-$ towards a continuous function $h^+$.
          By Lemma \ref{LemmeGenericallyInvariantGreenFunctions}, it is clear that ${u_n}_{|W_{V[\fin]}^-} \equiv 0$, thus
          $h^+_{W_{V[\fin]}^-} \equiv 0$.
          If $G^+ = h^+ + g_{(\sX, \sD^+)}$, then it is defined on $W^- \cap \cM_D^{\an}$ and satisfies $G^+ \circ
          s^{\an} = \lambda_1 G^+$.
        \end{proof}

          If $\sD^-$ is a model of $\theta_{f,X}^-$, we construct in the same fashion an open $s^{\an}$-invariant
          neighbourhood $U^+$ of $\bigcup_v \{p_+^v\}$ with $W^+ := X^{\an} \setminus U^+$ and the function $G^-$ which
          is defined over $W^+ \cap \cM_D^{\an}$, satisfies $G^- \circ (s^{-1})^{\an} = \lambda_1 G^-$ and is such that
          $G^- - g_{(\sX, \sD^-)}$ extends to a continuous function $h^-$ over $W^+$ that satisfies
          \begin{equation}
            h^-_{|W^+_{V[\fin]}} \equiv 0
            \label{<+label+>}
          \end{equation}
          Furthermore, we can choose $U^-$ and $U^+$ such that $U^- \Subset W^+$ and $U^+ \Subset W^-$. For $v \in
          V[\fin]$ we already have $U^-_v \cap U^+_v = \emptyset$. For the
        places outside $V[\fin]$ we can shrink $U^+_v, U^-_v$ such that it is the case. We also shrink $U^{\pm}_v$ for
          $v \not \in V[\fin]$ such that $G^+_{|U_v^+} \geq 1, G^-_{|U_v^-} \geq 1$, this
          is always possible because $G^\pm - g_{(\sX, \sD^\pm)}$ extends to a continuous function at $p_\pm$ and
          $\theta_{f,X}^{\pm}$ is effective.

          \subsection{Convergence everywhere}
          Define
          \begin{equation}
            \sU = \sX \setminus \left( \overline{X \setminus \cM_D} \cup \bigcup_{\m \not \in V} \sX_{\m} \right)
          \end{equation}
          where $\sX_\m = \sX \times_{\spec \OO_\K} \spec \OO_\K / \m$.
          This is a quasiprojective model of $\cM_D$ over $\OO_\K$. Let $\overline{\sD_0}$ be a model of $\theta_{f,X}^-$ and
          a boundary divisor of $\sU$
          in $\sX$. If $g_0$ is the Green function of $\overline{\sD_0}$ over $\cM_D^{\an}$ then we can suppose without loss
          of generality that for all $v \not \in V[\fin], g_{0, v} \geq 1$. We have already constructed the Green
          functions $G^\pm$ away from $p_\mp$.

          \begin{lemme}
            For every place $v \not \in V[\fin]$, over $U_v^- \cap \cM_D^{\an}$ the functions
            \begin{equation}
              \frac{G^-}{g_0}, \quad \frac{g_0}{G^-}
              \label{<+label+>}
            \end{equation}
            are continuous and bounded.
          \end{lemme}
          \begin{proof}
            With the notations of \S \ref{SubSecConvAwayFromPminus}, there exists a constant $A >0$ such that
            \begin{equation}
              -A \sD_0 \leq \sD^- \leq A \sD_0
              \label{<+label+>}
            \end{equation}
            and over $U_v^-$ we have by Proposition \ref{PropConvergenceSurWmoins} that
            \begin{equation}
              G^- = h^- + g_{(\sX, \sD^-)}
              \label{<+label+>}
            \end{equation}
            where $h^-$ is continuous and bounded. This shows the result.
          \end{proof}
    Let $\pi : \sY \rightarrow \sX$ be a resolution of indeterminacies of $s : \sX \dashrightarrow \sX$ and let $\sS : \sY
    \rightarrow \sX$ be a lift of $s$, write $S : Y \rightarrow X$ for the restriction to the generic fiber. Then, $h$
    is a green function of $\frac{1}{\lambda_1}S^* D - \frac{\mu}{\lambda_1}\pi^* D$. Therefore, there exists a constant
    $A$ such that
    \begin{equation}
      -A \overline{\sD_0} \leq \frac{1}{\lambda_1} S^* \overline \sD - \frac{\mu}{\lambda} \pi^* \overline \sD \leq A
      \overline{\sD_0}.
      \label{ }
    \end{equation}
    Thus, $\frac{h}{g_0} $ is a continuous bounded function over $\cM_D^{\an}$.
    We show that the sequence $u_n$ converges with respect to the boundary topology. Set the following constants
    \begin{equation}
      M_0 = \sup_{\cM_D^{\an}} \left|\frac{h}{g_0}\right|, \quad M_1 = \sup_{\cM^{\an}_{D, V[\fin]^c}}
       \frac{1}{\left|g_0\right|},  \quad M_2 = \sup_{W^-} \left| h \right|
     \end{equation}
     \begin{equation}
       M_3 = \sup_{U^-_{V[\fin]^c} \bigcap \cM_D^{\an}} \left|\frac{g_0}{G^-}\right|, \quad M_4 =
      \sup_{U^-_{V[\fin]^c} \bigcap \cM_D^{\an}} \left|\frac{G^-}{g_0}\right|
  \end{equation}
  where $V[\fin]^c$ is the set of places of $\K$ outside $V[\fin]$, including the archimedean ones.
    \begin{claim}
      Set $M := \max (M_2 M_1, M_0 M_3 M_4)$, then for every $k \geq 0$
      \begin{equation}
        -M g_0 \leq h \circ (s^{\an})^k \leq M g_0
        \label{EqInequality}
      \end{equation}
      over $\cM_D^{\an}$.
    \end{claim}
    \begin{proof}
      We will write $s$ instead of $s^{\an}$ to avoid heavy notations.
      Let $k \geq 0$ and $x \in \cM_D^{\an}$. Suppose first that $s^k (x) \in W^-$.
      If $x$ lies above $v \in V[\fin]$, then $h (s^k (x)) = 0$ by Lemma \ref{LemmeFonctionH} and \eqref{EqInequality} is
      obvious. Otherwise we have
      \begin{equation}
      \left|\frac{h(s^k(x))}{g_0 (x)}\right| \leq M_2 M_1
        \label{<+label+>}
      \end{equation}
      and \eqref{EqInequality} is satisfied.

      If $s^k(x) \not \in W^-$, then $x, s^k (x) \in U^- \subset W^+$. If $x$ lies above $V[\fin]$, then by
      Proposition \ref{PropConvergenceSurWmoins},,
      \begin{equation}
        G^-_{|W^+_{V[\fin]}} = g_{(\sX_{V[\fin]}, \sD^-_{V[\fin]})} = g_{0|V[\fin]},
      \end{equation}
      thus
      \begin{equation}
        \left|h(s^k (x))\right| \leq M_0 g_0 (s^k (x)) = \frac{M_0}{\lambda_1^k} g_0 (x).
        \label{<+label+>}
      \end{equation}

      Suppose $x$ does not lie above $V[\fin]$, let $y = s^k (x)$, then
      \begin{equation}
        \left|\frac{h(s^k(x))}{g_0 (x)}\right|= \left|\frac{h(y)}{g_0 (s^{-k} (y))}\right| \leq M_4 \left|\frac{h(y)}{ G^-
        (s^{-k}(y))}\right| = M_4 \left|\frac{h(y)}{\lambda_1^k G^- (y)}\right|.        \label{<+label+>}
      \end{equation}
      Thus,
      \begin{equation}
        \left|\frac{h(s^k(x))}{g_0}\right| \leq \frac{M_4}{\lambda_1^k} \left|\frac{g_0(y)}{G^-(y)}\right|
        \left|\frac{h(y)}{g_0(y)}\right| \leq \frac{M_0 M_3 M_4}{\lambda_1^k}
        \label{<+label+>}
      \end{equation}
    \end{proof}

    \begin{proof}[End of proof of Proposition \ref{PropConvergenceModelDivisorSemiSimple}]
    \begin{enumerate}
      \item If $\mu = \lambda_1$, then $u_n$ converges with respect to the boundary topology because by the claim
        \begin{equation}
          \sup_{\cM_D^{\an}} \left|\frac{h \circ s^k}{g_0} \right|
          \label{<+label+>}
        \end{equation}
        is the term of a converging sum.
      \item If $\left|\mu\right|< \left|\lambda_1\right|$, then $\left|\frac{u_n}{g_0}\right|$ converges uniformly towards $0$ over
        $\cM_D^{\an}$. Indeed,
        \begin{equation}
          \sup_{\cM_D^{\an}} \left| \frac{u_n}{g_0} \right| \leq M \sum_{\ell = 0}^{n-1} \frac{|\mu|^l}{ \lambda_1^{n-1}} \leq
          \frac{M}{\lambda_1^{n-1}} \frac{\left|\mu^n - 1\right|}{ \left|\mu - 1\right|} \xrightarrow[n \rightarrow +\infty]{} 0
          \label{<+label+>}
        \end{equation}
        if $\mu \neq \pm 1$ and otherwise
        \begin{equation}
          \sup_{\cM_D^{\an}} \left| \frac{u_n}{g_0} \right| \leq \frac{M n}{\lambda_1^{n-1}} \xrightarrow[n \rightarrow
          +\infty]{} 0.
          \label{<+label+>}
        \end{equation}
    \end{enumerate}
  \end{proof}

  \subsection{Proof of Theorem \ref{ThmInvariantAdelicDivisor}}\label{subsec:proof-thm-invariant-adelic-divisor}
  \begin{prop}\label{PropConvergencePourS}
    For any cyclic compactification $X$ of $\cM_D$ that satisfies Propostion \ref{PropDynamicsLoxodromicAutomorphism},
    let $s = f^{N_0}$ be an iterate that contracts $X \setminus \cM_D$. For any divisor $D \in \DivInf(X)_\R$ and any
    model adelic extension $\overline \sD$ of
    $D$, we have
    \begin{equation}
      \frac{1}{\lambda_1^N} (s^N)^* \overline \sD \rightarrow (D \cdot \theta^-) \overline
      \theta^+ (X).
      \label{EqConvergence2}
    \end{equation}
  \end{prop}
  \begin{proof} Using Proposition \ref{PropOperateurPullBack} (4) we write
    \begin{equation}
      D = \alpha \theta_{X}^+ + \beta D^- + R
      \label{<+label+>}
    \end{equation}
    where $E_+, F_+ \not \in \Supp R$ with the notations of Proposition \ref{PropOperateurPullBack}. Intersecting with
    $\theta^-$ and using Proposition \ref{PropOperateurPullBack}
    (1) and (3) we get $\alpha = D \cdot \theta^-$. For any model adelic extension $\overline \sD$ of $D$ we get by
    Proposition \ref{PropConvergenceModelDivisorSemiSimple} and Lemma \ref{LemmeConvergenceModelVerticalDivisor} that
    \begin{equation}
      \frac{1}{\lambda_1^N} (s^N)^* \overline \sD \rightarrow (D \cdot \theta^-) \overline \theta^+ (X).
      \label{<+label+>}
    \end{equation}

  \end{proof}

  \begin{prop}\label{prop:Convergence-avec-hypothese}
    Let $X$ be any cyclic completion of $\cM_D$ and let $D \in \DivInf(X)_\R$ such that $D \cdot \theta^- = 0$. For any
    model adelic extension $\overline \sD$ of $D$ we have
    \begin{equation}
      \frac{1}{\lambda_1^N} (f^N)^* \overline \sD \xrightarrow[N \rightarrow + \infty]{} 0.
      \label{eq:<+label+>}
    \end{equation}
  \end{prop}
  \begin{proof}
    We can assume up to blowing up that $X$ satisfies Proposition \ref{PropDynamicsLoxodromicAutomorphism}. Let
    $s = f^{N_0}$ be an iterate of $f$ that contracts $X \setminus \cM_D$. We have by Proposition
    \ref{PropConvergencePourS} that
    \begin{equation}
      \frac{1}{\lambda_1^{N_0 k}} (s^k)^* \overline \sD \xrightarrow[k \rightarrow + \infty]{} 0.
      \label{eq:<+label+>}
    \end{equation}
    This means that if $\overline \sD_0$ is a boundary divisor, there exists a sequence $\epsilon_k \rightarrow 0$ such
    that
    \begin{equation}
      - \epsilon_k \overline \sD_0 \leq \frac{1}{\lambda_1^{N_0 k}} (s^k)^* \overline \sD \leq \epsilon_k \overline \sD_0.
      \label{eq:convergence-Cauchy}
    \end{equation}
    We can assume without loss of generality that $- \overline \sD_0 \leq \overline \sD \leq \overline \sD_0$.
    Now, for every $N \geq 0$ write the Euclidian division of $N$ by $N_0, N = a_n N_0 + b_n$ and let $C > 0$ be a
    constant such that for every $b = 0, \dots, N_0 - 1$
    \begin{equation}
      0 \leq \frac{1}{\lambda_1^b} (f^b)^* \sD_0 \leq C \overline \sD_0.
      \label{eq:estimate-initial-term}
    \end{equation}
    We have
    \begin{equation}
      \frac{1}{\lambda_1^N} (f^N)^* \overline \sD = \frac{1}{\lambda_1^{N_0 a_N}} s^{a_N} \left(
      \frac{1}{\lambda_1^{b_N}} (f^{b_N})^* \overline \sD \right).
      \label{eq:<+label+>}
    \end{equation}
    Now since $f^*$ preserves effectiveness we have by \eqref{eq:convergence-Cauchy} and
    \eqref{eq:estimate-initial-term} that
    \begin{equation}
      -C \epsilon_{a_N} \overline \sD_0 \leq \frac{1}{\lambda_1^N} (f^N)^* \overline \sD \leq C \epsilon_{a_N} \overline
      \sD_0.
      \label{eq:<+label+>}
    \end{equation}
  \end{proof}

  \begin{cor}\label{cor:convergence-tout-le-monde-un-modele}
    Let $X$ be a cyclic completion of $\cM_D$. If $D \in \DivInf(X)_\R$ and $\overline \sD$ is a model adelic extension
    of $D$, then
    \begin{equation}
      \frac{1}{\lambda_1^N} (f^N)^* \overline \sD \xrightarrow[N \rightarrow + \infty]{} \left( D \cdot \theta^- \right)
      \overline \theta^+ (X).
      \label{eq:<+label+>}
    \end{equation}
  \end{cor}
  \begin{proof}
    We can assume that $X$ satisfies Proposition \ref{PropDynamicsLoxodromicAutomorphism}. Let $s = f^{N_0}$ be an
    iterate that contracts $X \setminus \cM_D$. We have by Proposition \ref{PropConvergencePourS} that
    \begin{equation}
      \frac{1}{\lambda_1^{N_0 k}} (s^k)^* \overline \sD \rightarrow (D \cdot \theta^-) \overline \theta^+ (X).
      \label{eq:convergence-iterate}
    \end{equation}
    Write the Euclidian division $N = a_N N_0 + b_n$, we have
    \begin{equation}
      \frac{1}{\lambda_1^N} (f^N)^* \overline \sD = \frac{1}{\lambda_1^{N_0 a_N}} (s^{a_N})^* \left(
      \frac{1}{\lambda_1^{b_N}} (f^{b_N})^* \overline \sD \right).
      \label{eq:convergence-Euclidian-division}
    \end{equation}
    Notice that $\frac{1}{\lambda_1^b} (f^b)^* D \cdot \theta^- = D \cdot \theta^-$. For $b = 0, \dots, N_0 - 1$, define
    \begin{equation}
      \overline \sD_b = \overline \sD - \frac{1}{\lambda_1^b} (f^b)^* \overline \sD.
      \label{eq:<+label+>}
    \end{equation}
    This is a model adelic divisor that satisfies $D_b \cdot \theta^- = 0$ , by Proposition
    \ref{prop:Convergence-avec-hypothese} we have
    \begin{equation}
      \frac{1}{\lambda_1^k} (s^k)^* \overline \sD_b \rightarrow 0.
      \label{eq:<+label+>}
    \end{equation}
    Thus, for every $b = 0 ,\dots, N_0 -1$ we have by \eqref{eq:convergence-iterate}
    \begin{equation}
      \frac{1}{\lambda_1^{N_0 k}} (s^k)^* \left( \frac{1}{\lambda_1^b} (f^b)^* \overline \sD \right) \rightarrow
      (D \cdot \theta^-) \overline \theta^+ (X)
      \label{eq:<+label+>}
    \end{equation}
    and this combined with \eqref{eq:convergence-Euclidian-division} shows the result.
  \end{proof}

  We now finish the proof of Theorem \ref{ThmInvariantAdelicDivisor} which follows from the following theorem.
  \begin{thm}\label{thm:invariant-adelic-divisor-and-convergence}
    The adelic divisor $\overline \theta^+ (X) =: \overline \theta^+$ does not depend on $X$. It is a strongly nef and
    effective adelic divisor and for any adelic divisor $\overline D \in \hatDivInf(\cM_D / \OO_\K)$ we have
    \begin{equation}
      \frac{1}{\lambda_1^N} (f^N)^* \overline D \rightarrow (D \cdot \theta^-) \overline \theta^+.
      \label{eq:<+label+>}
    \end{equation}
  \end{thm}

  \begin{proof}
    We show that the adelic divisor $\overline \theta^+ (X)$ defined in Corollary
    \ref{cor:convergence-tout-le-monde-un-modele} does not depend on the completion $X$. It suffices to show that
    $\overline \theta^+
  (X) = \overline \theta^+ (Y)$ for any compactification $Y$ above $X$. Let $D = \theta_{X}^+ \in \DivInf (X)$, then $\pi^* D
  \in \DivInf(Y)$ and satisfies $\pi^* D \cdot \theta_Y^- = \theta_{X}^+ \cdot \theta_{X}^- = \theta^+ \cdot \theta^- =
  1$. Thus we get by Corollary \ref{cor:convergence-tout-le-monde-un-modele} that $\overline \theta^+ (X) = \overline
  \theta^+ (Y)$.

  We now show that $\overline \theta^+$ is strongly nef and effective. Let $H$ be the ample divisor on $\overline \cM_D
  \subset \P^3$ defined by $H = \left\{ X = T = 0 \right\} + \left\{ Y =
  T = 0 \right\} + \left\{ Z = T = 0 \right\}$. Let $\overline H$ be a semipositive effective model of $H$ which
  exists by Lemma \ref{lemme:effective-semipositive-model}. Since $H$ is ample
  we have $H \cdot \theta^- > 0$. By Proposition \ref{PropFunctoriality} and Corollary
  \ref{cor:convergence-tout-le-monde-un-modele} we have that $\overline \theta^+$ is strongly nef and effective.

  Finally, let $\overline \sD \in \hatDivInf (\cM_D / \OO_\K)$ be an adelic divisor. For every $\epsilon > 0$, there
  exists a model adelic divisor $\sD_\epsilon$ such that
  \begin{equation}
    - \epsilon \overline \sD_0 \leq \overline D - \overline \sD_\epsilon \leq \epsilon \overline \sD_0.
    \label{eq:<+label+>}
  \end{equation}
  Since $D_\epsilon \cdot \theta^- \rightarrow D \cdot \theta^-$ and $f^*$ preserves effectiveness we have
  \begin{equation}
    \frac{1}{\lambda_1^N} (f^N)^* \overline D \rightarrow (D \cdot \theta^-) \overline \theta^+.
    \label{eq:<+label+>}
  \end{equation}
  \end{proof}

\begin{rmq}\label{rmq:semipositive-effective-canonical-model}
  Let $X$ be the compactification $\overline \cM_D \subset \P^3$ and $H$ the ample divisor $H = \left\{ X = T = 0
  \right\} + \left\{ Y = T = 0 \right\} + \left\{ Z = T = 0 \right\}$. Write $D = A / B$ where $A,B \in \OO_\K$, we have
  a natural model $(\sX, \sH)$ of $(X,H)$ given by
  \begin{align}
    \sX &= \Proj \OO_K [X,Y,Z,T] / \left( B T \left( X^2 + Y^2 + Z^2 \right) - B XYZ - A T^3  \right)\\
    \sH &=  \left\{ X = T = 0 \right\}_\sX + \left\{ Y = T = 0 \right\}_\sX + \left\{ Z = T = 0 \right\}_\sX
  \end{align}
  Applying Corollary \ref{cor:convergence-tout-le-monde-un-modele} with $\sH$ yields the definitions of the Green
  functions from the introduction.
\end{rmq}

\begin{prop}\label{PropPropertiesGreenFunctions}
  Let $G^+$ be the Green function of $\overline \theta^+$, then
  \begin{enumerate}
    \item \label{item:positivity} $G^+ \geq 0$.
    \item \label{item:invariance} $G^+ \circ f^{\an} = \lambda_1 G^+$.
    \item \label{item:boundedness--condition} $G^+ (x) = 0$ if and only if the forward $f^{\an}$-orbit of $x$ is bounded.
    \item \label{item:Green-function-extension} If $X$ is a compactification of $\cM_D$, then
      for any Green function $g$ of $\theta_{f,X}^+, G^+ - g$ extends to a
      continuous function over $X_v^{\an} \setminus \left\{ p_- \right\}$.
    \item \label{item:psh} If $v$ is archimedean, then $G^+_v$ is plurisubharmonic and pluriharmonic over the set
      $\left\{ G^+_v > 0 \right\}$.
  \end{enumerate}
\end{prop}
\begin{proof}
  \ref{item:positivity} follows from $\overline \theta^+$ being effective.
  \ref{item:invariance} follows from $f^* \overline \theta^+ = \lambda_1 \theta^+$.

  If $X$ is any compactification of $\cM_D$, we can suppose that it satisfies Proposition
  \ref{PropDynamicsLoxodromicAutomorphism} up to blowing up.
  Let $(\sX, \sD)$ be a model of $(X, \theta_{f,X}^+)$ for a compactification $X$ that satisfies Proposition
  \ref{PropDynamicsLoxodromicAutomorphism}.
  Fix a place $v$ of $\K$, let $U^+$ be the open neighborhood of $p_+$ in $X^{\an}$ constructed in \S
  \ref{SubSecConvAwayFromPminus}. Since $G^+ - g_{(\sX, \sD)}$ extends to a continuous function over $U^+$ and
  $\theta_{f,X}^+$ is effective, we can shrink $U^+_v$ such that $G^+_{U_v^+} > 0$.

  By \ref{item:invariance}, we have that $G^+ > 0$ over
\begin{equation}
  U_v := \bigcup_{n \geq 0} (f^{-n})^{\an} (U^+_v \cap \cM_D^{\an})
  \label{<+label+>}
\end{equation}
To show \ref{item:boundedness--condition}, let $x \in \cM_D^{\an}$ and let $v$ be the place over which $x$ lies. It
suffices to show that if the
forward orbit of $x \in (\cM_D)_v^{\an}$ is unbounded, then there exists $N_0$ such that
$(f^{N_0})^{\an}(x) \in U^+_v$.  Since $X_v^{\an}$ is compact, the sequence
$((f^n)^{\an}(x))$ must have an accumulation point $q \in X_v^{\an} \setminus (\cM_D)_v^{\an}$.

If $q \neq p_-^v$, then
since an iterate of $f : X \setminus \left\{ p_-^v \right\} \rightarrow X \setminus \{p_-^v\}$ contracts the complement of
$\cM_D$ to $p_+^v$ we
must have $(f^k)^{\an} (q) = p_+^v$ for some $k \geq 1$ thus by continuity there exists $N_0$ such that $(f^{N_0})^{\an}
(x) \in U^+_v$.

Otherwise $(f^n)^{\an} (x) \rightarrow p_-^v$.
Since $p_-^v$ is an attracting fixed point for $(f^{-1})^{\an}$, there exists a basis of neighbourhood $U^k_v$ of
$p_-^v$ in $X^{\an}_v$ such that for all $N$ large enough, $(f^{-N})^{\an} (U^k_v) \Subset U^k_v$ and we would get that
$x \in U^k_v$ for all $k \geq 0$, this is absurd.

To show \ref{item:psh}, let $H$ be the ample divisor on $\overline \cM_D \subset \P^3$ defined by
$\left\{ X = T = 0 \right\} +
\left\{ Y= T = 0 \right\} + \left\{ Z = T = 0 \right\}$ and let $\overline H$ be the semipostive effective model of
$H$ defined in Remark \ref{rmq:semipositive-effective-canonical-model} and $g_H$ the associated Green function. Suppose
$v$ is archimedean. Since $g_H$ is plurisubharmonic over $\cM_D
(\C), \frac{1}{\lambda_1^n} g_H \circ f^n$ also is. By local uniform
convergence, we get that $G^+$ is plurisubharmonic over $\cM_D (\C)$. To show the pluriharmonicity, it suffices to show
that $G^+_{|U^+_v \bigcap \cM_D(\C)}$ is pluriharmonic by the proof of (3). We can suppose that $g_H$ is pluriharmonic
over $U^+_v \bigcap \cM_D (\C)$ up to shrinking $U_v^+$, then since $U^+_v$ is $f^{\an}$-invariant,
$\frac{1}{\lambda_1^n} g_H \circ (f^{\an})^n$ is also pluriharmonic over $U^+_v \cap \cM_D (\C)$ and
$G^+_{|U^+_v \cap \cM_D (\C)}$ is pluriharmonic as the local uniform limit of pluriharmonic functions.

To show \ref{item:Green-function-extension}, Let $v$ be any place of $\K$, let $q \in X^{\an}_v \setminus
(\cM_D)_v^{\an}$ such that $q \neq p_-$. From
Proposition
\ref{PropDynamicsLoxodromicAutomorphism}, we can find an open neighbourhood $U^-_v$ of $p_-$ such that $q \not \in
U_v^-$ and $(f^{-1})^{\an} (U_v^-) \Subset U_v^-$. The proof of \S \ref{SubSecConvAwayFromPminus} shows that if $g_+$ is
a model Green function of $\theta_{f,X}^+$, then
\begin{equation}
  \frac{1}{\lambda_1^n} g_+ \circ (f^n)^{\an} - g_+
  \label{<+label+>}
\end{equation}
converges uniformly over $W_v^- = X^{\an}_v \setminus U_v^-$ to a continuous function which is equal to $G^+ - g_+$.
Since two Green functions of the same divisor differ by a continuous function we get the result for any Green function
of $\theta_{f,X}^+$.
\end{proof}

\begin{rmq}
  In \cite{chambert-loirFormesDifferentiellesReelles2012}, Chambert-Loir and Ducros developed a theory of plurisubharmonic
  functions, currents and differential forms on Berkovich spaces. Following their definitions, \ref{item:psh} also holds for
non-archimedean places with the same proof.
\end{rmq}

\section{Periodic points and equilibrium measure}\label{SecEquidistributionPeriodicPoints}

\subsection{Equidistribution}
Let $(x_n)$ be a sequence of $X (\overline K) \subset X (\overline \K_v)$ and let $\mu_v$ be a
measure on $X_v^{\an}$. We say that the Galois orbit of $(x_n)$ is equidistributed with
respect to $\mu_v$ if the sequence of measures
\begin{equation}
  \delta(x_n) := \frac{1}{\deg(x_n)} \sum_{x \in \Gal(\overline \K / \K) \cdot x_n} \delta_x
  \label{<+label+>}
\end{equation}
weakly converges towards $\mu_v$, where $\delta_x$ is the Dirac measure at $x$.

We say that a sequence of points $(x_n)$ of $X (\overline \K)$ is \emph{generic} if no subsequence
of $(x_n)$ is contained in a strict subvariety of $X$. In particular, a generic sequence is
Zariski dense.

\begin{lemme}\label{LemmeGenericSequence}
  Let $X$ be a projective variety over a number field $\K$ and let $(x_n)$ be a Zariski dense
  sequence of $X(\overline \K)$, then one can extract a generic subsequence of $(x_n)$.
\end{lemme}
\begin{proof}
  The set of strict irreducible subvarieties of $X$ is countable because $\K$ is a number field.
  Let $(Y_q)_{q \in \N}$ be the set of strict irreducible subvarieties of $X$. We construct a
  generic subsequence $(x_q')_{q \in \N}$ as follows. Set $Y'_q = \bigcup_{k \leq q} Y_k$. This is a
  strict subvariety of $X$. Let $n(1)$ be such that
  $x_{n(1)} \not \in Y_1 = Y_1'$ and suppose we have constructed $n(1) < \dots < n(q)$ such that $x_{n(i)} \not \in
  Y_i'$. Since $(x_n)$ is Zariski dense, there exists an integer $n(q+1) > n (q)$ such
  that $x_{n(q)} \not \in Y_q'$. This defines an increasing sequence $n(q)$ and we set $x_q ' = x_{n(q)}$, The sequence
  $(x_q ')$ is a subsequence of $(x_n)$ which is clearly generic.
\end{proof}

We will use the following arithmetic equidistribution theorem from Yuan and Zhang.
\begin{thm}[\cite{yuanAdelicLineBundles2023} Theorem 5.4.3]\label{ThmEquidistributionYuanZhang}
  Let $X$ be a quasiprojective variety over a number field $\K$ and let $\overline D$ be a
  nef adelic divisor over $X$ such that $D^{\dim X} > 0$. Let $(x_n) \in X
  (\overline \K)$ be a generic sequence such that $\lim_n h_{\overline D} (x_n) \rightarrow
  h_{\overline D} (X)$, then at every place $v$ the Galois orbit of the sequence $(x_n)$ is
  equidistributed with respect to the equilibrium measure $\mu_{\overline D, v}$ over $X^{\an}_v$.
\end{thm}

\subsection{Equidistribution of periodic points} Suppose $D \in \C$ is algebraic.
Let $f \in \Aut(\cM_D)$ be loxodromic. Let $\overline
\theta_f^+, \overline \theta_f^-$ be the two strongly nef adelic divisors provided by Theorem
\ref{ThmInvariantAdelicDivisor} for $f$ and $f^{-1}$. Recall that $\theta_f^+ \cdot \theta_f^- =1$. Set
\begin{equation}
  \overline \theta_f = \frac{\overline \theta_f^+ + \overline \theta_f^-}{2}.
\end{equation}
It is a strongly nef adelic divisor over $\cM_D$ and satisfies $\theta_f^2 = 1$. For every
place $v$, we write $\mu_{f,v}$ for the equilibrium measure of $\overline \theta_f$. We also write
$h_f := h_{\overline \theta_f}$ and $G_f = \frac{G^+_f + G^-_f}{2}$ for the Green function of $\overline \theta_f$.

\begin{thm}\label{ThmYuan}
  If $(p_n)$ is a generic sequence of $\cM_D (\overline \K)$ of periodic points of $f$, then for every
  place $v$ of $\K$ the Galois orbit of $(p_n)$ is equidistributed with respect to the measure $\mu_{f, v}$ over
  $\cM_{D,v}^{\an}$.
\end{thm}

\begin{proof}
  We apply Yuan-Zhang's equidistribution theorem to the adelic divisor
  $\overline \theta$. We need to show that the sequence
  $h_{f}(p_n)$ converges to $h_{f} (\cM_D)$. Since the points $p_n$ are periodic, we have for all $v$, $G_v (p_n^v) = 0$ by
  Proposition \ref{PropPropertiesGreenFunctions} (3). Thus we need to show that $h_{f} (\cM_D) = 0$. To
  do so we apply Theorem 5.3.3 of \cite{yuanAdelicLineBundles2023}. Namely, let
  \begin{equation}
    e(\cM_D, \overline \theta_f) := \sup_{U \subset \cM_D} \inf_{p \in U} h_{f} (p)
    \label{<+label+>}
  \end{equation}
  where $U$ runs through open subsets of $\cM_D$. This quantity is called the \emph{essential minimum} of $\overline
  \theta$. Since we have a generic sequence of periodic points, we get $e (\cM_D, \overline \theta_f) = 0$. Theorem 5.3.3 of
  \cite{yuanAdelicLineBundles2023} states that
  \begin{equation}
    e (\cM_D, \overline \theta_f) \geq h_f (X).
    \label{<+label+>}
  \end{equation}
  Therefore we get $h_f (\cM_D) = 0$ and Yuan's equidistribution theorem gives the desired result.
\end{proof}

\begin{cor}
  If $D$ is algebraic and $f,g$ are two loxodromic automorphisms of $\cM_D$ such that $\Per(f) \cap
  \Per(g)$ is Zariski dense, then for every place $v$ of $\Q(D)$,
  \begin{equation}
    \mu_{f,v} = \mu_{g,v}
    \label{<+label+>}
  \end{equation}
\end{cor}
\begin{proof}
  Let $(x_n)$ be a Zariski dense sequence of $\Per(f) \cap \Per(g)$. By Lemma
  \ref{LemmeGenericSequence}, we can suppose that $(x_n)$ is generic. By Theorem \ref{ThmYuan}, for
  every place $v$ of $\K$, the Galois orbit of the sequence $(x_n)$ equidistributes with respect to
  both the measures $\mu_{f,v}$ and $\mu_{g,v}$. Thus, they must be equal.
\end{proof}

\subsection{Ergodic properties of the equilibrium measure}\label{subsec:ergodic-prop-equilibrium-measure}
Let $D \in \C$ and consider the embedding $\Q(D) \hookrightarrow \C$. This corresponds to an archimedean place of
$\Q(D)$. Let $X$ be a compactification of $\cM_D$ (over $\C$) that satisfies Proposition \ref{PropDynamicsLoxodromicAutomorphism}. The
measure $\mu_f = dd^c G^+_f \wedge dd^c G^-_f$ that we construct in this paper is exactly the measure associated to the
birational transformation $f : X \dashrightarrow X$ constructed in \cite{bedfordEnergyInvariantMeasures2005}. Indeed,
The quantitative condition (3) in loc.cit is satisfied because the indeterminacy points of $f$ are intersection points
of $X \setminus \cM_D$. In particular the main theorem of
\cite{bedfordEnergyInvariantMeasures2005} implies the following.
\begin{thm}\label{thm:ergodic-properties-equilibrium-measure}
  For every $D \in \C$, for every archimedean place $v$ of $\Q(D)$, the measure $\mu_{f,v}$ is ergodic, mixing and hyperbolic.
\end{thm}

A \emph{saddle} point is a point where the differential of
$f$ does not have any eigenvalue of modulus 1.

\begin{cor}\label{cor:supp-measure-in-closure-saddle-periodic-points}
  For every archimedean place $v$, $\Supp \mu_{f,v}$ is contained in the closure of the saddle periodic points.
\end{cor}
\begin{proof}
  This follows from Theorem 4.2 of \cite{katokLyapunovExponentsEntropy1980} because $\mu_{f,v}$ is ergodic, hyperbolic
  and not supported on a finite orbit.
\end{proof}

\section{Saddle periodic points are in the support of the equilibrium
measure}\label{SubSecSaddlEPeriodicPointsAreINSupport}
Let $D \in \C$ be algebraic. Fix an archimedean place $v$ of $\Q(D)$. For this section we work only over this place so
we will drop the index $v$ and fix an embedding $\Q(D) \hookrightarrow \C$.

\begin{thm}\label{ThmPeriodicPointsInSupportOfMeasure} Let $f$ be a loxodromic automorphism of $\cM_D$. The support of
  the measure $\mu_f$ is the closure of the saddle periodic points of $f$.
\end{thm}

By Corollary \ref{cor:supp-measure-in-closure-saddle-periodic-points}, it suffices to show that every saddle periodic
point is in the support of $\mu_f$.
This theorem, stated in~\cite{cantatDynamiqueAutomorphismesSurfaces2001} and
\cite{cantatBersHenonPainleve2009}, follows directly from the work of Dinh and Sibony
in~\cite{dinhRigidityJuliaSets2013}, which extends \cite{bedfordPolynomialDiffeomorphismsC21991a},
and an argument of \cite{bedfordPolynomialDiffeomorphismsOfC1993} for Hénon type automorphisms of
the complex affine plane. A sketch of proof is given in \cite[\S 3.1 and 3.2]{cantatBersHenonPainleve2009}. We provide a
more detailed proof here.

\subsection{Green functions and bounded orbits} First, let us summarize some of the properties of the
function $G^+_f: \cM_D(\C)\to \R^+$ of $\overline \theta_f^+$ from Proposition \ref{PropPropertiesGreenFunctions}.
The
\emph{stable manifold} of a point $q \in \cM_D (\C)$ is the set of
points $p$ such that
\begin{equation}
  \dist (f^n(q), f^n(p)) \xrightarrow[n \rightarrow + \infty]{} 0.
  \label{<+label+>}
\end{equation}
Fix a compactification $X$ of $\cM_D$ that satisfies Proposition \ref{PropDynamicsLoxodromicAutomorphism}. We have
\begin{itemize}
  \item[(a)] $\{G^+_f=0\}$ coincides with the
    set $K^+(f)$ of points with a bounded forward orbit;
  \item[(b)] $G^+_f$ is plurisubharmonic, and
    is pluriharmonic on the set $\{G^+_f>0\}$;
  \item[(c)] the set $K^+(f)$ is closed in $\cM_D(\C)$,
    its closure in $X(\C)$ coincides with $K^+(f)\cup \{p_-\}$;
  \item[(d)] locally, near
    every point $q\neq p_-$ of $X(\C) \setminus \cM_D (\C)$, \begin{equation} G^+_f(x)=-\sum_i a_i
    \log |z_i| +u(x) \end{equation} where the functions $z_i$ are holomorphic
    equations of the boundary components containing $q$, the real numbers $a_i\geq 0$ are the weight
    of $\theta_{f,X}^+$, and $u(x)$ is a continuous (pluriharmonic) function.
  \item[(e)] there is an
    open neighborhood $U^-$ of $p_-$  in $X(\C)$ such that $f^{-1}(U^-) \Subset U^-$ and $U^-$ is
    contained in the basin of attraction of $p_-$ for the backward dynamics;
    there is an open neighborhood  $U^+$ of $p_+$ with similar properties for $f$ instead of
    $f^{-1}$;
  \item[(f)] If $q$ is a saddle periodic point, its stable manifold $W^s(q)$ is
    contained in $K^+(f)$; in fact, the proof of Proposition 5.1
    in~\cite{bedfordPolynomialDiffeomorphismsC21991} shows that $W^s(q)$ is
    contained in the boundary of $K^+(f)$;
  \item[(g)] $f$ does not preserve any algebraic curve $C_0\subset \cM_D(\C)$.
\end{itemize} In particular, if $S$ is a closed positive current supported by $\overline{K(f)}$,
then its support does not intersect the open set $U^-$.

\subsection{Rigidity of $\overline{K^+(f)}$ and equidistribution of stable manifolds} The
properties (a)
to (g) are sufficient to apply the arguments of Sections~4, 5, 6 of~\cite{dinhRigidityJuliaSets2013}. More precisely,
one first obtains Theorem 6.6 of~\cite{dinhRigidityJuliaSets2013}, because its proof relies only on the above
properties and general results concerning closed positive currents (in particular Corollary 3.13
of~\cite{dinhRigidityJuliaSets2013}). \footnote{The only changes in this proof are that (1) $\P^2(\C)$ should be
  replaced by $X(\C)$ and the line at infinity by $X \setminus \cM_D$; and (2) the function
  $\log(1+\parallel z\parallel^2)^{1/2}$ should be replaced by a smooth Green function associated to the
$\R$-divisor $\theta_{f,X}^+$.}
Then, one gets directly the following fact (which corresponds to a weak version of Theorem 6.5
of~\cite{dinhRigidityJuliaSets2013}, with the same proof):

\begin{thm}\label{ThmRigidityJuliaSets} The set $\overline{K^+(f)}$ (resp. $\overline{K^-(f)}$) supports a
  unique closed positive current, namely $T^+_f=dd^cG^+_f$ (resp. $T^-_f$) up to multiplication by a
  positive constant.
\end{thm}
Consequently, we get the following result: {\sl{Given any algebraic
    curve $C_0\subset \cM_D$, the sequence of currents $ \lambda_1(f)^{-n} \{(f^n)^*C_0\}$ converges towards
a positive multiple of $T^+_f$ as $n$ goes to $+\infty$}} (see Corollary 6.7
of~\cite{dinhRigidityJuliaSets2013}). Thus, $T^+_f$ can be approximated by a sequence of currents of
integration on algebraic curves of a fixed genus (properly renormalized); in this context, one can apply
the theory of strongly approximable laminar currents, as developed by Dujardin
(see~\cite{cantatDynamicsAutomorphismsCompact2014, dujardinLaminarCurrentsBirational2004,
dujardinStructurePropertiesLaminar2005} for an introduction and \S \ref{SubSecLaminarity}).

This rigidity results provides automatic equidistribution theorems for $(1,1)$ positive currents. We
shall need the following specific application.

If $q$ is a saddle periodic point of $f$, then its stable manifold $W^s(q)$ is biholomorphic to the
complex line \footnote{Indeed, it is a Riemann surface, it is homeomorphic to $\R^2$, and $f$ acts
  on it as a contraction fixing $q$, so $W^s(q)$ cannot be a disk and Riemann uniformization theorem
says that it is a copy of $\C$}. Denote by $\xi\colon \C\to W^s(q)\subset \cM_D(\C)$ a one to one
holomorphic parametrization of $W^s(q)$; $\xi$ is an entire holomorphic curve. To such a curve, one
can associate a family of currents of mass $1$, constructed as follows. One fixes a Kähler form
$\kappa$ on $X(\C)$ and one measures lengths, areas and volumes with respect to this form. For
instance, if $\mathbb{D}_r\subset\C$ is the disk of radius $r$ centered at the origin, then
\begin{equation} Area(\xi(\mathbb{D}_r))=\int_{\xi(\mathbb{D}_r)}\kappa =
\int_{\mathbb{D}_r}\xi^{*}\kappa \end{equation} is the area of the image of $\mathbb{D}_r$ by $\xi$.
Averaging with respect to $dr/r$, one introduces the function \begin{equation}
N(R)=\int_{t=0}^RArea(\xi(\mathbb{D}_t))\frac{dt}{t}. \end{equation} Now, for each disk
$\mathbb{D}_r$, one can consider the current of integration over $\xi(\mathbb{D}_r)$: to a smooth form
$\alpha$ of type $(1,1)$, this current $\{ \xi(\mathbb{D}_r) \}$  associates the number
\begin{equation} \langle \{ \xi(\mathbb{D}_r) \}\vert \alpha \rangle = \int_{\xi(\mathbb{D}_r)}\alpha=
\int_{\mathbb{D}_r}\xi^*\alpha. \end{equation} Taking averages with respect to the weight $dr/r$ one
obtains the following family of currents, parametrized by a radius $R>0$: \begin{align} \langle
  N_\xi(R)\vert \alpha \rangle & =   \frac{1}{N(R)}\int_{t=0}^R   \langle \{ \xi(\mathbb{D}_r) \}\vert
  \alpha \rangle\; \frac{dt}{t} \\ & =  \frac{1}{N(R)}\int_{t=0}^R \int_{\xi(\mathbb{D}_R)}\alpha \;
\frac{dt}{t}. \end{align} The normalization by $1/N(R)$ assures that the mass $\langle N_\xi(R) \vert
\kappa\rangle$ is equal to $1$ for every $R>0$. From an inequality of Ahlfors, and from the compactness of
the space of positive currents of mass $1$,  there are sequences of radii $(R_n)$ such that $N_\xi(R_n)$
converges to a closed positive current $S$.  A priori, such a closed positive current $S$ depends on the
choice of the sequence $R_n$; if there is a unique closed positive current $S$ that can be obtained as such
a limit, one says that there is a unique Ahlfors-Nevanlinna current (namely $S$) associated to $\xi$.

\begin{cor}[Proposition 4.10, Corollary 4.11 \cite{dinhRigidityJuliaSets2013}] Let $q$ be a saddle
  periodic point of
  $f$. Let $\xi\colon \C\to \cM_D(\C)$ be a holomorphic   parametrization of the stable manifold of $f$.
  Then, there is a unique  Ahlfors-Nevanlinna current associated to $\xi$, and this current is equal to
$T^+_f$. \end{cor}

\subsection{Laminarity, Pesin theory and consequence} \label{SubSecLaminarity}

The measure $\mu_f=T^+_f\wedge T^-_f$ is an
ergodic measure of positive (and maximal) entropy for $f$, and tools from Pesin theory can be used
to describe the dynamics of $f$ with respect to this measure. In particular, in our setting, one can
apply the work of Bedford, Lyubich, and Smillie in~\cite{bedfordPolynomialDiffeomorphismsOfC1993} or
the work of Dujardin in~\cite{dujardinLaminarCurrentsBirational2004}.

\begin{dfn}
  \begin{enumerate}
    \item A family of disjoint horizontal graphs $\Gamma$ in $\D \times \D$ is called a flow box. If it is
      equipped with a measure $\lambda$ on $\{0\} \times
      \D$ we call it a \emph{measured flow box}. It defines a closed positive current $T_{\Gamma, \lambda}$
      in $\D \times \D$ defined by
      \begin{equation}
        \langle T_{\Gamma, \lambda}, \alpha \rangle = \int_{a \in \D} \int_{\Gamma_a} \alpha
        \dd\lambda(a).
        \label{<+label+>}
      \end{equation}

    \item A current $T$ is \emph{uniformly laminar} if for every $x \in \Supp T$, there exist an open subset
      $V \ni x$ such that $V$ is biholomorphic to a bidisk and a measured flow box $(\Gamma, \lambda)$
      in $V$ such that $T_{|V} = T_{\Gamma, \lambda}$.

    \item A current is \emph{laminar} if there exists a family of disjoint measured flow boxes $(\Gamma_i,
      \lambda_i)$ such that
      \begin{equation}
        T = \sum_i T_{\Gamma_i, \lambda_i}.
        \label{<+label+>}
      \end{equation}

    \item A current is \emph{strongly approximable} if it is the weak limit of a sequence of integration
      currents $\frac{1}{d_n} [C_n]$ such that
      \begin{equation}
        genus(C_n) + \sum_{p \in \Sing(C_n)} n_p (C_n) = O (d_n).
        \label{<+label+>}
      \end{equation}
    \item A current is \emph{diffuse} if it does not charge algebraic curves.
  \end{enumerate}
\end{dfn}
The main result of \cite{dujardinStructurePropertiesLaminar2005} is that {\sl{if $X$ is a projective
    rational surface and $T$ is a strongly approximable diffuse current on $X$, then $T$ is laminar and for every
flow box $\Gamma$, $T_{|\Gamma}$ is uniformly laminar.}} The discussion after Theorem
\ref{ThmRigidityJuliaSets} shows that $T_f^+$ and $T_f^-$ are strongly approximable currents. They are
also diffuse by Proposition 6.3 of \cite{dillerDynamicsBimeromorphicMaps2001}.
\begin{dfn}
  If $S_1, S_2$ are two uniformly laminar diffuse currents with a representation
  \begin{equation}
    S_i = \int_{A_i} [D_{a,i}] d\mu_i (a)
    \label{<+label+>}
  \end{equation}
  then we define the \emph{geometric
  intersection} of $S_1, S_2$ as
  \begin{equation}
    S_1 \geom S_2 := \int_{A_1} \int_{A_2} [D_{a,1} \cap D_{b,2}] d\mu_1(a) \otimes d \mu_2 (b)
    \label{<+label+>}
  \end{equation}
  where $[D_{a,1} \cap D_{b,2}]$ is the sum of Dirac masses at the intersection point if the
  intersection is finite and 0 otherwise. We extend the definition of geometric intersection to sums of
  uniformly laminar currents by taking geometric intersection with respect to each term of the sum.
  We say that a product is \emph{geometric} if $S_1 \wedge S_2 = S_1 \geom S_2$.
\end{dfn}

\begin{dfn}
  A Pesin box is a pair $(U,K)$ where $U$ is an open subset isomorphic to a bidisk $\D \times \D$ and
  a compact $K \subset U$ of positive $\mu_f$-measure such that
  \begin{enumerate}
    \item Every point in $K$ is a hyperbolic point of $f$.
    \item The local stable and unstable manifolds of the points of $K$ are vertical and horizontal
      graphs in $U$.
    \item For all pair of distinct points $(x,y) \in K^2, W^s_{loc}(x) \cap W^u_{loc}(y)$ is a
      singleton contained in $K$.
  \end{enumerate}
\end{dfn}
In particular, the local stable and unstable manifolds define a lamination $K^s$ and $K^u$ in $U$. By
the main theorem of \cite{dujardinStructurePropertiesLaminar2005}, ${T_f^+}_{|K^s}$ is uniformly
laminar so there exists a transverse measure $\lambda_K^+$ such that
\begin{equation}
  {T_f^+}_{|K^s} = T_{K^s, \lambda_K^+}.
  \label{<+label+>}
\end{equation}

\begin{thm}[Theorem 1, Theorem 5.2 of \cite{dujardinLaminarCurrentsBirational2004}]
  The currents $T_f^+, T_f^-$ are diffuse strongly approximable and therefore laminar and the laminar
  structure is compatible with Pesin theory.
  The current $T_f^+$ is equal to
  \begin{equation}
    T_f^+ = \sum_{(U,K)} T_{K^s, \lambda_K^+}
    \label{<+label+>}
  \end{equation}
  and since the potentials of $T_f^\pm$ are continuous, the product $T_f^+ \wedge T_f^-$ is
  geometric. Thus, the measure $\mu_f$ has a product structure with respect to the laminations
  induced by the local stable and unstable manifolds.
\end{thm}

\begin{proof}[Proof of Theorem \ref{ThmPeriodicPointsInSupportOfMeasure}]
  We apply the following argument taken from Section 9
  of~\cite{bedfordPolynomialDiffeomorphismsOfC1993} and \cite{dujardinLaminarCurrentsBirational2004}
  \S 5.2. Pick a saddle
  periodic point $q$ of $f$, take a small neighborhood $W$ of $q$, and consider its stable manifold,
  parametrized by $\xi\colon \C\to W^s(q)$. Take a Pesin box $(U,K)$. Since the Ahlfors-Nevanlinna
  current of $\xi$ coincides
  with $T^+_f$, each disk of $K^s$ is a limit of disks $\xi(D_i)$, for some topological
  disks $D_i\subset \C$. Now, $\mu_f(K) > 0$ and by the product structure we have
  \begin{equation}
    {\mu_f}_{|K} = {T_f^+}_{|K^s} \geom {T_f^-}_{|K^u}.
    \label{<+label+>}
  \end{equation}
  Thus $T_f^+$ and $T_f^-$ give mass to $K^s$ and $K^u$ respectively.
  Since the laminations $K^u$ and $K^s$ intersect
  transversaly, one finds a disk $\xi(D_i)$ that intersects $K^u$ transversaly. Then, if one
  applies $f^{N}$ with $N$ large, the preimages of $\xi(D_i)\cap K^u$ approach the point
  $q$, and the Inclination lemma (or Lambda lemma see \cite{palisGeometricTheoryDynamical1982} \S 7.1)
  implies that the images of the leaves of $K^u$ are (very
  large) disks which, in the neighborhood $W$ of $q$, converge towards $W^u(q)$ (in the $C^1$
  topology). We thus obtain a sequence of uniformly laminar currents
  \begin{equation}
    \frac{1}{\lambda_f^n} (f^n)_* \left( {T_f^-}_{|K^u} \right)_{|W} \leq T_f^-.
    \label{<+label+>}
  \end{equation}
  Doing the same with the unstable manifold $W^u(q)$ and the dynamics of $f^{-N}$, one pulls
  back $ K^s$ near $q$. Thus we get
  the sequence of measures
  \begin{equation}
    \mu_n := \frac{1}{\lambda_f^{2n}} \left( (f^n)^* {T^+_f}_{|K^s} \right)_{|W} \geom \left(
    (f^n)_* {T_f^-}_{|K^u} \right)_{|W} \leq \mu_f
    \label{<+label+>}
  \end{equation}
  that gives mass to $W$.
  Since this work for any neighborhood of $q$, this
  point is in the support of $\mu_f$ and the closure of the saddle periodic points are contained in the support of
  $\mu_f$. Thus, Theorem~\ref{ThmPeriodicPointsInSupportOfMeasure} is proven.

\end{proof}

\section{Proof of Theorem \ref{BigThmRigidity}}
The proof of Theorem \ref{BigThmRigidity} relies on the following proposition.

\begin{prop}\label{PropSaddlePeriodicPointBOundedOrbit}
  Let $f \in \Aut(\cM_D)$ be a loxodromic automorphism with $D = 0$ or $D = 2 - 2 \cos \left(
  \frac{\pi}{q} \right)$ and let $v$ be an archimedean place. Then, $f$ admits a periodic saddle
  fixed point $q(f) \in \cM_D (\C)$ such that
  \begin{enumerate}
    \item $q (f) \in \Supp (\mu_{f,v})$
    \item If $g \in \Aut(\cM_D)$ is loxodromic such that $f$ and $g$ do not share a common iterate,
      then $(g^n (q(f)))$ is unbounded.
  \end{enumerate}
\end{prop}
Item (1) follows from Theorem \ref{ThmPeriodicPointsInSupportOfMeasure}.

Assuming the proposition, suppose that $f,g$ share a Zariski dense subset of periodic points, we can
suppose that they share a generic sequence of periodic points by Lemma \ref{LemmeGenericSequence}. Then
by Theorem \ref{ThmYuan} we have equality of the equilibrium measures of $f$ and $g$ at every
place so in particular at every archimedean place. Fix $v$ one of them. Suppose that $f$ and $g$
do not share a common iterate, then $(g^n(q(f)))_n$ is unbounded. Let $\mu = \mu_{f,v} =
\mu_{g,v}$.  Since $ \Supp \mu = \Supp \mu_{f,v} = \Supp \mu_{g,v}$, we have that $\Supp \mu$ is a compact subset of $\cM_D (\C)$ invariant
by $f$ and $g$. Since $q(f) \in \Supp \mu_{f,v} = \Supp \mu$ we get that $(g^n(q(f))) \subset
\Supp \mu$ which is a contradiction.

\subsection{Construction of the saddle fixed point $q(f)$}
Suppose first that $D = 0$.
Up to taking an iterate of $f$ we can suppose that there exists a loxodromic element $\Phi_f
\in \SL_2 (\Z)$ such that $f = f_{\Phi_f}$. Denote by $p(f) = p (f_{\Phi_f})$ and $q(f)
= q(f_{\Phi_f})$ the fixed points constructed using Minsky theorem. These two fixed point are saddle
fixed points by \cite{mcmullenRenormalization3ManifoldsWhich1996} Corollary 3.19. The fixed point
$q(f)$ corresponds to a representation $\rho_\infty : F_2 \rightarrow
\PSL_2 (\C)$, one can show that $\rho_\infty$ also satisfies Theorem \ref{ThmHyperbolisation} even
though the punctured torus is not compact. One can also show that for any automorphism $g$ of $\cM_0$ the differential of $g$
at $(0,0,0)$ has order 1 or 2, thus $p(f), q(f) \neq (0,0,0)$ and it is a smooth point of $\cM_0$.

Suppose now that $D = 2 -2\cos \frac{\pi}{q}$. Following \S \ref{subsec:orbifold}, for any $\Phi_f \in \SL_2 (\Z)$ we
define $q(f) = \rho_\infty (\Phi_f) \in \cM_D$.
\subsection{The sequence $(g^n(q(f)))$ is unbounded}

Suppose $D = 0$ we can consider $S$ as the flat torus $T=\R^2/\Z^2$ with a puncture at
the origin, i.e. $S=T\setminus\{0\}$, or as a complete hyperbolic surface $X$ of finite area (we fix
  such a hyperbolic structure, it corresponds to some point $X$ in the Teichmüller space
$Teich(S)\simeq \mathbb{D}$).

We refer to \S 1.4 and \S 1.5 of \cite{otalTheoremeDhyperbolisationPour1996} for the defintions of measured laminations
and the classifications of elements in $\Mod(S)$ for $S$ a real compact surface with negative Euler characteristic. An
element $f$ of $\Out^+(F_2)$ is pseudo-Anosov if the corresponding matrix $A_f\in\SL_2(\Z)$ has
$\Tr(A_f)^2 > 4$. In that case, the matrix has two eigenvalues $\lambda (f) > 1$ and $1 / \lambda
(f) < 1$ and the mapping class is represented by a linear automorphism of
the torus $T$ (fixing the origin $o$) with stable and unstable linear foliations. In the hyperbolic
surface $X$, these foliations give rise to two measured laminations $F_-$ and $F_+$ (by geodesic
lines). If $C\subset S$ is a closed curve (represented by some geodesic in $X$), one can define two
intersection numbers $i(C, F_+)$ and $i(C,F_-)$; they depend only on the free homotopy class of $C$.
The product $j(C)=i(C, F_+) i(C,F_-)$ is $f$-invariant, because $f$ stretches $F_+$ by a dilatation
factor $\lambda(f) >1$, and contracts $F_-$ by $1/\lambda(f)$;  if $C$ is not  homotopic to a loop
around the puncture $j(C)$  is strictly positive (any closed geodesic is transverse to $F_+$ and
$F_-$).

If $D = 2 - 2 \cos(\pi / q)$, let $S$ be the genus one torus with an orbifold singularity of order
$q$. We have seen that there exists a characteristic finite covering $\tilde S \rightarrow S$ with
$\tilde S$ a compact surface of negative Euler characteristic. We let $X =
\HH^2/ \Gamma$ be a hyperbolic surface homeomorphic to $\tilde S$ (i.e $X \in \Teich(\tilde S)$).
If $f \in \Out^+ (F_2)$ is pseudo-Anosov then it lifts to a pseudo-Anosov $\tilde f \in \Mod (X) =
\Out^+ (F_2)$ pseudo-Anosov also. In that case, there exist two measured laminations $F_+$ and
$F_-$ over $\tilde S$ (the stable and the unstable one) Proposition 1.5.1 of
\cite{otalTheoremeDhyperbolisationPour1996}. We have that for any geodesic $\gamma \in \tilde S$,
\begin{equation}
  \frac{(\tilde f)^{\pm k}_* \gamma}{\ell ((\tilde f)_*^{\pm k}\gamma)} \xrightarrow[k \rightarrow
  +\infty]{} F_\pm
  \label{EqConvergenceVersLamination}
\end{equation}
in the sense of measured laminations. (This also holds in the case $D = 0$). Here $\ell$ is the
length induced by the hyperbolic structure
from the quotient $\HH^2 / \Gamma$ so $\ell((\tilde f)^{\pm k}_* \gamma)$ grows like
$\lambda(\tilde f)^k$. We also have that $j(\gamma) = i(\gamma, F_+) i(\gamma, F_-)$ is
$f$-invariant as $i(\tilde f_* (\gamma) , F_\pm) = \lambda(\tilde f)^{\mp 1} i(\gamma, F_\pm)$ and
if $\gamma$ is a geodesic, then $j(\gamma) > 0$. To unify the notations we will still denote by $f$
the lift $\tilde f$ of $f$ to $X$.

\begin{lemme} If $f$ and $g$ are two loxodromic elements of $\Out^+(F_2)\simeq SL_2(\Z)$ generating a
  non-elementary subgroup of $SL_2(\Z)$, then given any geodesic $\gamma \subset X$, $j(g^n(\gamma))$ goes
  to $+\infty$ as $n$ goes to $+\infty$.
\end{lemme}

\begin{proof} Let $G_+$ and $G_-$ be the unstable and stable laminations associated to $g$ in $X$.
  Since $f$ and $g$ generate a non-elementary subgroup of $GL_2(\Z)$, $G_+$ is
  transverse to both $F_+$ and $F_-$ (equivalently, the four fixed points of $A_f$ and $A_g$ on
  $\P^1(\R)$ are distinct). Thus, by Equation \eqref{EqConvergenceVersLamination} $j(g^n(C)) \simeq
  \lambda(g)^n i(G_+, F_+) i(G_-, F_-)$ by continuity of the intersection number (see
  \cite{otalTheoremeDhyperbolisationPour1996} p.151).
\end{proof}

\begin{lemme} Let $f$ and $g$ be two loxodromic elements of $\Out^+(F_2)\simeq \SL_2(\Z)$ generating a
  non-elementary subgroup of $SL_2(\Z)$. Let $\gamma \subset X$ be a geodesic, and let $[\gamma]$ be its
  free homotopy class. Then the sequence $g^n[\gamma]$ intersects each orbit of $f$ only finitely many
times. \end{lemme}

\begin{proof} This follows from the previous lemma and the fact that $j(\cdot)$ is $f$-invariant so
  it is constant in each orbit of $f$.
\end{proof}

Recall the definition of $M_{\Phi_f}$, $\tilde M_{\Phi_f}$, $\rho_\infty$ and $\alpha_f$ from
Theorem \ref{ThmHyperbolisation} (here we consider $f \in \Mod (\tilde S)$ if we are in the
orbifold case).
In $M_{\Phi_f}$, the number of simple closed geodesics
of length $\leq L$ is finite (for every $L>0$); thus, in $\tilde{M}_{\Phi_f}$, given any upper bound
$L$, there are only finitely many homotopy classes of simple closed curves {\sl{up to the action of
$f^\Z$}} (Note that, since $\alpha_f$ acts by isometry, each closed geodesic $C\subset \tilde{M}_f$
gives rise to infinitely many geodesics $\alpha_f^n(C)$ with the exact same length).
\paragraph{\bf Proof of Proposition \ref{PropSaddlePeriodicPointBOundedOrbit} (2)}
Fix a generator $a$ in $\pi_1 (S)$ where $S$ is either the punctured torus or the genus 1 torus with
an orbifold singularity of index $q$. Set $k$ to be the index of $\pi_1(\tilde S)$ in $\pi_1 (S)$.
and $k=1$ otherwise. The element $a^k$ gives rise to a closed
geodesic $A$ in $\tilde{M}_{\Phi_f}$. From these preliminaries and the previous lemma, the sequence
of homotopy classes $g^n(a^k)$ corresponds to a sequence of closed geodesics in $\tilde{M}_{\Phi_f}$,
with length going to infinity because $f$ acts by isometry on $\tilde M_{\Phi_f}$.

Now, $g^n(a^k)$ corresponds to a (conjugacy class of a) matrix $\rho_\infty(g^n(a^k))\in \SL_2(\C)$, and the
trace of this matrix is related to the length of the geodesic by a simple formula; in particular,
the fact that the length goes to infinity implies that the modulus of the trace goes to $+\infty$.
Since for any matrix $A \in \SL_2 (\C), \Tr A^k$ is a polynomial in $\Tr A$ we get that $\Tr (\rho_\infty(
g^n (a)))$ goes to infinity.
This implies that the orbit of $q(f)$ under the action of $g$ on $\cM_D(\C)$ is discrete, going to
infinity.

\section{For a transcendental parameter $D$}
We finish this paper by proving Theorem \ref{BigThmTranscendentalParameter} which we restate.
\begin{thm}
  Let $D \in \C$ be transcendental and let $f,g \in \Aut(\cM_D)$ be loxodromic automorphisms. The
  following assertions are equivalent:
  \begin{enumerate}
    \item $\Per(f) = \Per(g)$.
    \item $\exists N,M \in \Z \setminus \{0\}, f^N = g^M$.
  \end{enumerate}
\end{thm}

\begin{proof}
  We can suppose that $f,g \in \SL_2(\Z)$, for any parameter $w \in \C$, we denote by $f_w$ the
  automorphism induced by $f \in \SL_2(\Z)$ over $\cM_w$. For the parameter $w = 0$, we have
  constructed a hyperbolic fixed point $q(f) \in \cM_0 (\C)$. Because $q(f) \neq (0,0,0)$, it a smooth point of
$\cM_0 (\C)$ and we can find local analytic
coordinates $u,v,w$ at $q(f) \in \C^3$ with $w = \kappa+2$ (recall the notations from the introduction) such that $f_w$ is locally of the form
  \begin{equation}
    f_w (u,v,w) = (\lambda u, \frac{1}{\lambda} v, w)
    \label{<+label+>}
  \end{equation}
  where $\lambda, \frac{1}{\lambda}$ are the two eigenvalues of the differential of $f_0$ at
  $q(f)$. By the analytic implicit function theorem, there exists $\epsilon > 0$ a local analytic curve
  $c_\epsilon : w \in \D(0, \epsilon) \mapsto c_\epsilon (w) \in \cM_w$ such that
  $c_\epsilon(0) = q(f)$ and $f_w (c_\epsilon(w)) = c_\epsilon(w)$. Now, if $f,g$ do not share a
  common iterate, then the orbit of $q(f)$ under $g_0$ is unbounded by Proposition
  \ref{PropSaddlePeriodicPointBOundedOrbit}. Thus, for all $k \in \Z$, we have $g_0^k (q(f)) \neq
  q(f)$. We show the following.

  \begin{lemme}
    If $D \in \C$ is transcendental, then for all $k \in \Z_{\geq 0}$, there exists $p \in \cM_D (\C)$ such that
    \begin{equation}
      f_D(p) = p \text{ and } g_D^\ell (p) \neq p, \forall 1 \leq \ell \leq k.
      \label{Eq3}
    \end{equation}
  \end{lemme}
  Using the lemma, we can conclude because $f_D$ admits a finite number of fixed points since $f_D$
  does not admit an invariant curve, thus we must have $\Per (f_D) \neq \Per (g_D)$.
\end{proof}
  \begin{proof}[Proof of the lemma]
    Notice that this statement does not depend on the transcendental parameter $D$. Indeed, let $D'$
    be another transcendental parameter, then there exists a Galois automorphism $\sigma \in
    \Gal(\C / \overline \Q)$ that exchange $D$ and $D'$. Since the family of surfaces $\cM_D$ gives a
    foliation of $\C^3$ we can view a point $p \in \cM_D (\C)$ as a point in $\C^3$ and apply $\sigma$
    to each coordinates, we denote by $p^\sigma \in \C^3$ the new obtained point. We apply $\sigma$
    to \eqref{Eq3} to get
    \begin{equation}
      f_{\sigma(D)} (p^\sigma) = p^\sigma \text{ and } g_{\sigma(D)}^\ell (p^\sigma) \neq p^\sigma,
      \forall 1 \leq \ell \leq k.
      \label{<+label+>}
    \end{equation}
    Now, fix $k \geq 1$. For any transcendental parameter $t$ small enough we have
    \begin{equation}
      f_t (c_\epsilon(t)) = c_\epsilon(t)
      \label{<+label+>}
    \end{equation}
    by construction of the curve $c_\epsilon$ and
    \begin{equation}
      \forall 1 \leq \ell \leq k, g_t^k (c_\epsilon(t)) \neq c_\epsilon(t)
      \label{<+label+>}
    \end{equation}
    by continuity since we have $g_0^m (q(f)) \neq q(f)$ for all $m \in \Z$. Thus, the lemma
    is shown.
  \end{proof}

\bibliographystyle{alpha}
\bibliography{biblio}
\end{document}